\documentclass[12pt]{amsart}
\usepackage{amssymb,latexsym}
\usepackage{enumerate}

\makeatletter
\@namedef{subjclassname@2010}{%
  \textup{2010} Mathematics Subject Classification}
\makeatother

\newtheorem{theo}{Theorem}[section]
\newtheorem{lem}[theo]{Lemma}
\newtheorem{prop}[theo]{Proposition}
\newtheorem{cor}[theo]{Corollary}
\newtheorem{lemma}[theo]{Lemma}

 \theoremstyle{definition}
\newtheorem{definition}[theo]{Definition}

\newtheorem{example}[theo]{Example}
 \theoremstyle{remark}

 \numberwithin{equation}{section}

\newtheorem{remark}[theo]{Remark}

\newcommand{\betheo}{\begin{theo}$\!\!\!${\bf } }
\newcommand{\entheo}{\end{theo}}

\newcommand{\becor}{\begin{cor}$\!\!\!$  }
\newcommand{\encor}{\end{cor}}

\newcommand{\belem}{\begin{lem}$\!\!\!${\bf .} }
\newcommand{\enlem}{\end{lem}}

\newcommand{\beprop}{\begin{prop}$\!\!\!${\bf .} }
\newcommand{\enprop}{\end{prop}}

\newcommand{\bedefi}{\begin{definition}$\!\!\!$ \rm }
\newcommand{\findefi}{ \end{definition}}

\newcommand{\beex}{\begin{example}$\!\!\!$ \rm }
\newcommand{\enex}{ \end{example}}

\newcommand{\berem}{\begin{remark}$\!\!\!$ \rm }
\newcommand{\enrem}{ \end{remark}}

\numberwithin{equation}{section}

\newcommand{\be}{\begin{equation}}
\newcommand{\en}{\end{equation}}

\newcommand{\bea}{\begin{eqnarray}}
\newcommand{\ena}{\end{eqnarray}}

\newcommand{\beano}{\begin{eqnarray*}}
\newcommand{\enano}{\end{eqnarray*}}

\newcommand{\bee}{\begin{enumerate}}
\newcommand{\ene}{\end{enumerate}}

\newcommand{\bei}{\begin{itemize}}
\newcommand{\eni}{\end{itemize}}

\newcommand{\betab}{\begin{tabular}}
\newcommand{\entab}{\end{tabular}}

\newcommand{\bd}{\begin{displaymath}}

\newcommand{\ad}{^{\mbox{\scriptsize $\dag$}}}
\newcommand{\das}{^{\dag {\rm\textstyle *}}}

\newcommand{\up}{\raisebox{0.7mm}{$\upharpoonright $}}%
\newcommand{\cci}{\raisebox{1.3pt}{$\,\scriptscriptstyle\bullet\,$}}

\def\A{{\mathfrak A}}
\def\B{{\mathcal B}}

\def\D{{\mathcal D}}
\def\E{{\mathcal E}}

\def\H{{\mathcal H}}
\def\I{{\mathcal I}}

\def\L{{\mathcal L}}
\def\M{{\mathcal M}}

\def\P{{\mathcal P}}
\def\R{{\mathcal R}}

\def\J{\relax\ifmmode {\mathcal J}\else${\mathcal J}$\fi}
\def\x{\relax\ifmmode {\mbox{*}}\else*\fi}

\def\AG{{\mathfrak A}}

\def\MM{{\mathfrak M}}
\def\NG{{\mathfrak N}}

\newcommand{\mc}{\mathcal}
\newcommand{\mb}{\mathbb}
\newcommand{\Ao}{{\AG}_o}
\newcommand{\norm}[1]{\| #1 \|}

\newcommand{\mult}{{\scriptstyle \Box}}

\newcommand{\wmult}{{\scriptscriptstyle \Box}}
\newcommand{\qsmult}{\cci}
\newcommand{\ha}{^{\rm\textstyle *}}

\newcommand{\noi}{\noindent}

\newcommand{\pa}{partial \mbox{*-algebra}}

\newcommand{\po}{partial O\mbox{*-algebra}}

\newcommand{\tpa}{topological partial \mbox{*-algebra}}

\newcommand{\LD}{{\L}\ad(\D)}
\newcommand{\LDH}{{\L}\ad(\D,\H)}
\newcommand{\LDHO}[2]{{\L}\ad(#1,#2)}
\newcommand{\LBDH}{{\L}_b\ad(\D,\H)}

\newcommand{\w}{\rm w}
\newcommand{\Alb}{\A_{\sf{lb}}}
\newcommand{\nlb}[1]{\left\|#1\right\|_{\sf{lb}}}

\newcommand{\ssp}{_{\scriptscriptstyle \oplus}}

\def\dag{\dagger}

\newcommand{\vp}{\varphi}

\newcommand{\ip}[2]{\left\langle {#1}\left| {#2}\right.\right\rangle}

\def\OL{\relax\ifmmode {\sf L}\else{\textsf L}\fi}
\def\OR{\relax\ifmmode {\sf R}\else{\textsf R}\fi}
\newcommand{\ZOL}{\OL^\circ}
\newcommand{\dd}{\D}
\newcommand{\hh}{\H}

\begin{document}
\baselineskip=17pt

\title[Bounded elements in topological partial *-algebras]{Bounded elements in certain topological partial *-algebras}

\author[J.-P. Antoine]{Jean-Pierre Antoine}
\address{%
Institut de Recherche en Math\'ematique et Physique \\
Universit\'e Catholique de Louvain\\
B-1348   Louvain-la-Neuve\\
Belgium}
\email{jean-pierre.antoine@uclouvain.be}

\author[C. Trapani]{Camillo Trapani}
\address{%
Dipartimento di Matematica e Informatica \\
Universit\`a di Palermo\\
I-90123 Palermo\\
Italy}
\email{trapani@unipa.it}

\author[F. Tschinke]{Francesco Tschinke}
\address{%
Dipartimento di Metodi e Modelli Matematici \\
 Facolt\`a d'Ingegneria - Universit\`a di Palermo\\
I-90128 Palermo\\
Italy}
\email{ftschinke@unipa.it}

\begin{abstract} \noindent
We continue our study of topological partial *algebras,  focusing our attention to the interplay between the various
partial multiplications. The special case of partial *-algebras of operators is examined first, in particular the link between
the strong and the weak multiplications, on one hand, and  invariant  positive sesquilinear (ips) forms, on the other.
 Then the analysis is extended to abstract \tpa s, emphasizing the crucial role played by appropriate bounded elements, called
$\M$-bounded. Finally, some remarks are made concerning representations in terms of the so-called partial GC*-algebras of operators.
\end{abstract}



\maketitle

\section{Introduction}

Studies on partial *-algebras have provided so far a considerable amount of information about their representation
theory and their structure. In particular, many results have been obtained for \emph{concrete} partial *algebras,
i.e., partial *-algebras of closable operators (the so-called partial O*-algebras). A full analysis of these aspects has
been developed by Inoue and two of us and it can be found in the monograph \cite{ait_book}, where earlier articles are quoted.

In a recent paper \cite{antratsc}, we have started the analysis of spectral properties of partial *-algebras and, in particular, partial O*-algebras.
We continue this study in the present work, focusing rather on the interplay between the different partial multiplications at hand.
Indeed, the main feature of partial O*-algebras is that they carry two natural
multiplications, the weak one and the strong one. Even tough they are, in general, \pa s only with respect to the first one,
the interplay of the two multiplications allows a rather natural definition of inverse of an element and thus a good starting point for the spectral theory.
These two ingredients (the possibility of defining a \emph{strong} multiplication and the existence of bounded elements)
are then introduced in the abstract context leading to the notion of \emph{topologically regular} \pa. This, in turn, suggests to characterize a special class of \tpa s,
called \emph{partial GC*-algebras}, both in an abstract version and in an operator version, i.e., a special class of \po s.

In the case of a partial O*-algebra $\A$, the best situation for the spectral theory occurs
when $\A$ contains sufficiently many \emph{bounded elements}, i.e., bounded operators. The same property
will show up here. We will characterize the appropriate notion of bounded elements, namely, the so-called $\M$-bounded elements.
The very name shows that the construction derives from a (sufficiently large) family  $\M$ of invariant  positive sesquilinear (ips) forms.
As a matter of fact, the strong partial multiplication is also derived from this family, and so are the associated spectral results.
For instance, an element $x\in\A$ has a finite spectral radius if and only if it is $\M$-bounded.
As a result, the whole picture becomes coherent.

The notion of bounded element of a topological *-algebra was first proposed by Allan in 1965 \cite{allan} with the goal of developing a spectral theory for these algebras.
Allan's definition was applied to O*-algebras by  Schm\"udgen \cite{schm_bounded}, but he did not include the topic in his monograph \cite{schmu}.
Bounded elements in purely algebraic terms have been considered by Vidav \cite{vidav} and Schm\"udgen \cite{schm_weyl} with respect to some (positive) cone.
This ingenious approach seems to be unfit for general partial *-algebras,  since they may fail to possess a natural positive cone. Of course, if the locally convex partial *-algebra $\A$ contains
a dense *-algebra (like the $\Ao$-regular partial *-algebras considered in Section \ref{sec:Mbdd}), then it has a natural positive cone, namely, the closure of the positive cone of $\Ao$.
However,  we will not pursue in this direction here.
Finally, Cimpri\v{c} defines a notion of  element of a *-ring bounded with respect to a given module. His construction, albeit in a totally different context, presents some analogy
with the one we describe in Section  \ref{sec:Mbdd}, in particular the C*-seminorm used in Proposition \ref{prop_algebra}.
\medskip

The paper is organized as follows. After some preliminaries  about \pa s (Section \ref{sect_preliminaries}), taken mostly from \cite{ait_book} and \cite{antratsc},
we  discuss in Section \ref{sec:pmult-ips}
the interplay between the partial multiplications and sets of ips-forms. We show, in particular, how the strong partial multiplication on
the space $\LDH$ may be characterized in terms of ips-forms.
Then, in Section \ref{sec:Mbdd}, which is the core of the paper, we show how a sufficient family  $\M$ of ips-forms leads one to the appropriate notion of $\M$-bounded elements
 and of the strong partial multiplication
induced by $\M$. The corresponding spectral elements are defined and they are shown to behave as expected. Finally, in Section \ref{sec:represent},
we make some remarks on representations. In particular, we examine under which conditions a partial GC*-algebra may have a faithful representation by a partial GC*-algebra of operators,
that is, a representation in some space $\LDH$.
It is worth mentioning that the family $\M$ of ips-forms defines in a locally convex partial *-algebra a cone of positive elements, making possible a generalization
 Schm\"udgen's approach in \cite{schm_weyl} to the present framework. We leave this investigation to a future paper.

\section{Preliminaries}
\label{sect_preliminaries}

For general aspects of the theory of \pa s and of their representations, we refer to the monograph
\cite{ait_book}.
For the convenience of the reader, however, we   repeat here the essential definitions, following the
definitions and notations given there.

 First we   recall
that a \pa\ $\A$ is a complex vector space with conjugate linear
involution  $\ha $ and a distributive partial multiplication
$\cdot$, defined on a subset $\Gamma \subset \A \times \A$,
satisfying the property that $(x,y)\in \Gamma$ if, and only if,
$(y\ha ,x\ha )\in   \Gamma$ and $(x\cdot y)\ha = y\ha \cdot x\ha $.
From now on we will write simply $xy$ instead of $x\cdot y$ whenever
$(x,y)\in \Gamma$. For every $y \in \A$, the set of left (resp.
right) multipliers of $y$ is denoted by $L(y)$ (resp. $R(y)$), i.e.,
$L(y)=\{x\in \A:\, (x,y)\in \Gamma\}$, resp. $R(y)=\{x\in \A:\, (y,x)\in \Gamma\}$. We denote by $L\A$ (resp.
$R\A$)  the space of universal left (resp. right) multipliers of
$\A$.

In general, a \pa\ is not associative, but in several situations a weaker form of associativity holds. More precisely, we say
that $\A$ is \emph{semi-associative} if $y \in R(x)$ implies $yz\in R(x)$, for every $z \in R\A$, and
 $$
(xy)z=x(yz).
$$

Throughout this paper we will only consider partial *-algebras with unit: this means that there exists an element $e\in \A$ such that $e=e\ha$, $e\in R\A\cap L\A$
and $xe=ex=x$, for every $x\in \A$.

Let $\H$ be a complex Hilbert space and $\D$ a dense subspace of $\H$.
 We denote by $ \L\ad(\D,\H) $
the set of all (closable) linear operators $X$ such that $ {D}(X) = {\D},\; {D}(X\x) \supseteq {\D}.$ The set $
\L\ad(\D,\H ) $ is a  \pa\
 with respect to the following operations: the usual sum $X_1 + X_2 $,
the scalar multiplication $\lambda X$, the involution $ X \mapsto X\ad := X\x \up {\D}$ and the \emph{(weak)}
partial multiplication $X_1 \mult X_2 = {X_1}\ad\x X_2$, defined whenever $X_2$ is a weak right multiplier of
$X_1$ (we shall write $X_2 \in R^{\rm w}(X_1)$ or $X_1 \in L^{\rm w}(X_2)$), that is, whenever $ X_2 {\D} \subset
{\D}({X_1}\ad\x)$ and  $ X_1\x {\D} \subset {\D}(X_2\x).$

It is easy to check that $X_1 \in L^{\rm w}(X_2)$ if and only if there exists $Z \in \LDH$ such that
\begin{equation} \label{altwp}
\ip{X_2\xi}{X_1\ad \eta} = \ip{Z\xi}{\eta}, \quad \forall \xi, \eta \in \D.
\end{equation}
In this case $Z= X_1 \mult X_2$.
  $\LDH$ is neither associative nor semi-associative.
If $I$ denotes the identity operator of $\H$, we put $I_\D=I\up\D$. Then $I_\D$ is the unit of the partial *-algebra $ \L\ad(\D,\H)$.

If $\NG \subseteq \LDH$ we denote by $R^{\w}\NG$ the set of right multipliers of all elements of $\NG$. We recall that
 $$
R^{\rm w}\LDH= \{A\in \LDH:\, A\mbox{ is bounded and } A:\D \to \D\ha \},
$$
 where $$
\D\ha =\bigcap_{X\in \LDH}D({X\das}).
$$

We denote by $\LBDH$ the bounded part of $\LDH$, i.e., $\LBDH=\{X \in \LDH : X$  is a bounded operator$\}=\{X \in \LDH : \overline{X}\in {\mc B}(\H)\}$.

{A $\ad $-invariant} subspace $\MM$ of $\LDH$ is called a \emph{(weak) partial O*-algebra} if $X\mult Y \in \MM$, for every
$X, Y \in \MM$ such that $X \in L^{\rm w}(Y)$. $\LDH$ is the maximal partial O*-algebra on $\D$.

 The set $\LD:=\{X\in \LDH:\, X, X\ad:\D \to \D$\} is a *-algebra; more precisely, it is the maximal O*-algebra on $\D$
(for the theory of O*-algebras and their representations we refer to \cite{schmu}).

\medskip
Some interesting classes of partial O*-algebras (such as partial GW*-algebras) can be defined with help of
  certain topologies on $\LDH$ and its commutants.

The \emph{weak topology} ${\sf t}_w$ on $\LDH$ is  defined by the seminorms
$$ r_{\xi, \eta}(X)=|\ip{X\xi}{\eta}|, \quad X \in \LDH,\, \xi, \eta \in \D.$$
The \emph{strong topology} ${\sf t}_s$ on $\LDH$ is  defined by the seminorms
$$ p_\xi(X)=\|X\xi\|, \quad X \in \LDH, \, \xi \in \D.$$

\noi The \emph{strong* topology} ${\sf t}_{s^\ast}$ on  $\LDH$ is usually defined by the seminorms
 $$p^*_\xi (X)= \max\{\|X\xi\|, \|X\ad\xi\|\}, \,  \xi \in \D.$$

\noi If $\NG$ is a $\ad$-invariant subset of $\LDH$, the \emph{weak unbounded commutant} of $\NG$ is defined by
$$
\NG_\sigma'=\{Y\in \LDH:\, \ip{X\xi}{Y\ad\eta}=\ip{Y\xi}{X\ad\eta}, \; \forall\, X\in \NG, \xi,\eta\in \D\}.
$$
The {\em weak bounded commutant} $\NG_{\w}'$ of $\NG$  is defined by $\NG_{\w}'=\{Y \in \NG_\sigma':\, Y \mbox{ is bounded}\}.$

If $\NG$ is a \po, the {\em quasi-weak bounded commutant} $\NG_{{\rm q}\w}'$ of $\NG$ is defined as follows.
$$\NG_{{\rm q}\w}'= \{C\in \NG_{\w}':\, \ip{CX\ad\xi}{Y\ad\eta}=\ip{C\xi}{(X\mult Y) \eta}, \; \forall\, X\in L(Y), \xi,\eta\in \D\}.
$$

If $\NG$ is an O*-algebra of
bounded operators on $\D$, then $\NG''_{\w\sigma}= \overline{\NG}^{{\sf t}_{s^\ast}}$. This statement applies, in particular, to the set $\mathfrak{P}:= \{X\in \LBDH:\, X,
X\ad: \D \to \D\}$, which is an O*-algebra of bounded operators on $\D$ (it is in fact the bounded part of $\LD$) and $\mathfrak{P} \subset R^{\rm w}\LDH$. Then
$\mathfrak{P}''_{\w\sigma}= \overline{\mathfrak{P}}^{{\sf t}_{s^\ast}}$. The fact that  {$\mathfrak{P}'_{\w}={\mb C} I_\D$}, , implies that
$\overline{\mathfrak{P}}^{{\sf t}_{s^\ast}}=\LDH$ and, thus, $R^{\rm w}\LDH$ is
${\sf t}_{s^\ast}$-dense in $\LDH$.

\medskip
In $\LDH$ we can consider also the so-called \emph{strong multiplication} $\circ $. It is defined in the following
way:
\begin{equation}\label{eqn_strongmult}
\left\{\begin{array}{l}
X\circ Y \mbox{ is well-defined if } \; Y:\D\to D(\overline{X}), \; X\ad: \D \to D(\overline{Y\ad})\\
(X\circ Y)\xi= \overline{X}(Y\xi), \quad \forall\, \xi \in \D.
\end{array} \right.
\end{equation}

We shall write $Y \in R^{\rm s}(X)$ (or $X\in L^{\rm s}(Y)$). In general, this strong multiplication does not make
$\LDH$ into  a partial *-algebra, since the distributive property fails. However, a subspace $\MM$ of
$\LDH$ may happen to be a partial *-algebra with respect to the strong multiplication. In this case we say, as in
\cite{ait_book}, that $\MM$ is a strong partial O*-algebra.

\vspace{2mm} A \emph{*-representation} of a  \pa\ $\A$ in the
Hilbert space $\H$ is a linear map $\pi : \A \rightarrow\L\ad(\D,\H)$     such that:
(i) $\pi(x\x) = \pi(x)\ad$ for every $x \in \A$; (ii) $x \in L(y)$
in $\A$ implies $\pi(x) \in L^{\rm w}(\pi(y))$ and $\pi(x) \mult\pi(y) = \pi(xy).$
 The *-repres\-entation  $\pi$ is said to be \emph{bounded} if $\overline{\pi (x)} \in {\mc B}(\H)$ for every $x \in\A$.

\vspace{2mm} Let $\varphi$ be a positive sesquilinear form on
$D(\varphi) \times D(\varphi)$, where $D(\varphi)$ is a
subspace of $\A$. Then we have
\begin{align}
\varphi(x,y) &= \overline{\varphi(y,x)}, \ \ \ \forall \, x, y \in
D(\varphi),
\\
 |\varphi(x,y)|^2 &\leqslant \varphi(x,x) \varphi(y,y), \ \ \
\forall \, x, y \in D(\varphi). \label{2.2}
\end{align}
We put
\[
N_\varphi= \{ x \in D(\varphi) : \varphi(x,x)=0\}. \] By
\eqref{2.2}, we have
\[
N_\varphi= \{ x \in D(\varphi) : \varphi(x,y)=0, \ \ \ \forall \,
y \in D(\varphi) \},
\]
and so $N_\varphi$ is a subspace of $D(\varphi)$ and the quotient
space $D(\varphi) / N_\varphi := \{ \lambda_\varphi(x) \equiv
x + N_\varphi ; x \in D(\varphi) \}$ is a pre-Hilbert space with
respect to the inner product $\ip{\lambda_\varphi(x)}
{\lambda_\varphi(y)} = \varphi(x, y), x,y \in D(\varphi)$. We
denote by $\H_\varphi$ the Hilbert space obtained by completion of
$D(\varphi) / N_\varphi$.

\medskip
A positive sesquilinear form $\vp$  on $\A \times \A$ is said to be \emph{invariant}, and called an \emph{ips-form}, if
there exists a subspace $B(\varphi) $
of $\A$ (called a {\it core} for $\varphi$) with the properties
\begin{itemize}
\item[({\sf ips}$_1$)] $B(\varphi) \subset R\A$ ;

\item[({\sf ips}$_2$)] $\lambda_\varphi(B(\varphi))$ is dense in $\H_\varphi$ ;

\item[({\sf ips}$_3$)]  $\varphi(ax, y) = \varphi(x, a\x y), \, \forall \, a \in \A, \forall \, x,y \in B(\varphi)$ ;

\item[({\sf ips}$_4$)] $\varphi(a\x x, by) = \varphi(x, (ab)y), \ , \forall \, a \in L(b), \forall \, x,y \in B(\varphi)$.
\end{itemize}
In other words, an ips-form is an \emph{everywhere defined} biweight, in the sense of \cite{ait_book}.

 To every ips-form $\vp$ on $\A$, with core $B(\varphi) $, there corresponds a triple $(\pi_\vp, \lambda_\vp, \H_\vp)$, where $\H_\vp$ is a Hilbert space,
$\lambda_\vp$ is a linear map from $B(\varphi) $ into $\H_\vp$ and $\pi_\vp$ is a *-representation on $\A$ in the
Hilbert space $\H_\vp$. We refer to  \cite{ait_book} for more details on this celebrated GNS construction.

\medskip
 Let $\A$ be a \pa\ with unit $e$. We assume that $\A$ is a locally convex Hausdorff vector space under the topology
$\tau$ defined by a (directed) set $\{p_\alpha\}_{\alpha \in \I}$ of seminorms. Assume that\footnote{Condition ({\sf cl}) was called
({\sf t1}) in \cite{antratsc}.}
\begin{itemize}

\item[({\sf cl})] for every $x \in \A$, the linear map $\OL_x: R(x)\to \A$ with $\OL_x(y)=xy$, $y\in R(x)$,
is closed with respect to $\tau$, in the sense that, if $\{y_\alpha\}\subset R(x) $ is a net such that $y_\alpha \to y$   and $xy_\alpha \to z \in \A$, then $y\in R(x)$
and $z=xy$. 

 \end{itemize}

\noi Starting from the family of seminorms $\{p_\alpha\}_{\alpha \in
\I}$, we can define a second topology $\tau\ha $ on $\A$ by
introducing the set of seminorms $\{ p\ha _\alpha(x)\}$, where
$$
p\ha _\alpha(x)= \max\{p_\alpha(x), p_\alpha(x\ha )\}, \quad x \in \A.
$$
The involution $x\mapsto x\ha $  is automatically
$\tau\ha $-continuous. By ({\sf cl}) it follows that, for every $x \in
\A$, $\OL_x$ is $\tau\ha $-closed. And it turns out that the
map $\OR_y: x\in L(y)\mapsto xy\in \A$ is also  $\tau\ha $-closed.

\medskip

If $\Ao$ is a $\tau\ha $-dense subspace of  $R\A$,
then the restriction $\OL_x{\up{\Ao}}$ to $\Ao$ of the map $\OL_x$ is $\tau$-closable. Let us denote by
$\ZOL_x$ its $\tau$-closure defined on the following subspace of $\A$:
$$
D(\ZOL_x)=\{y \in \A: \exists \{y_\alpha\}\subset \Ao,
y_\alpha\stackrel{\tau}{\to} y; \, xy_\alpha \stackrel{\tau}{\to} z \in
\A\}.
$$

In terms of the latter, we may define  a new multiplication
$\qsmult$  on $\A$ by
$$
\left\{
\begin{array}{ll}
 y \in R_{\Ao}(x) \Leftrightarrow y \in D(\ZOL_x) \mbox{ and } x\ha  \in D(\ZOL_{y\ha })\\
x{\qsmult} y:=\ZOL_x y=\lim_\alpha
({\OL_x}{\up{\Ao}})y_\alpha.
 \end{array} \right.
$$
 We refer to the multiplication $\qsmult$ as the strong multiplication {\em induced} by $\Ao$. Clearly, $R_{\Ao}(x)\subset R(x)$, i.e.,
  if $x\qsmult y$ is well-defined, then $y \in R(x)$ and $x\qsmult y=xy$. On the other
hand, if $y\in R(x)$, $x\qsmult y$ need not be defined. The
definition itself implies that   $x\qsmult y$ is well-defined if,
and only if, $y\ha \qsmult x\ha $ is well-defined and one has
$$
 (x\qsmult y)\ha  = y\ha  \!\qsmult x\ha .
$$

We remark that, in general $\qsmult$ does not make   $\A$ into  a partial
*-algebra, since the distributive law may fail.

Let $\A$ be a partial *-algebra with unit $e$ and assume that $\A$ is a locally convex
space with respect to a given topology $\tau$. Then $\A$ is called
\emph{topologically regular} if it satisfies ({\sf cl}) and $R\A \cap
L\A$ contains a \emph{distinguished} *-algebra $\Ao$, i.e., $\Ao$
is a $\tau\ha $-dense *-subalgebra of $\A$ (containing the unit $e$)  such that, for the
multiplication $\qsmult$ induced by $\Ao$, the following
associative law holds, for all $x,y,z \in \A$:

\medskip
\centerline{ if $z \in R(y), yz\in R(x)$ and $y\in R_{\Ao}(x)$, then $z \in R(x\qsmult y),$ }

\noi and
\begin{equation}\label{assoc}
x(yz)=(x \qsmult y)z.
\end{equation}

 In particular the following semi-associativity with respect to $\Ao$ holds:
if $x\cci y$ is well-defined, then $x\cci(y b)$ is well defined for every $b\in\Ao$ and
$$
(x\cci y)b=x(y b),
$$
which follows easily from \eqref{assoc}.
\medskip

An element $a\in \A$ of a topologically regular \pa\ $\A$ is called \emph{left $\tau$-bounded} if there exists
$\gamma_a>0$ such that
\be\label{ltau-bdd}
p_\alpha (ax)\leq \gamma_a p_\alpha (x), \quad \forall\, x\in R\A, \,
\forall\, \alpha \in \I.
\en
 The set of left $\tau$-bounded elements of $\A$ is denoted by $\Alb$.
In general, $x\in\Alb$ does {\it not} imply that $x\ha \in\Alb$.
For $a \in \Alb$ we put
$$
\nlb{a}= \sup\{p_\alpha(ax): \alpha\in \I, \,x\in R\A,\,
p_\alpha(x)=1\}.
$$
It is easily seen that $\nlb{\cdot}$ is a norm on $\Alb$ \cite{antratsc}.

\medskip

A topologically regular \pa\  $\A$ with a
distinguished \mbox{*-subalgebra} $\Ao$ is called a
\emph{partial GC*-algebra} if
\begin{itemize}
\item[(i)] $\A$ is $\tau\ha $-complete;
\item[(ii)] $\Ao\subset \Alb$ and $\Ao$ is $\tau\ha $-dense in $\A$;
\item[(iii)]$\Alb$ is a C*-algebra with respect to the norm $\nlb{\cdot}$.
\end{itemize}

\section{Partial multiplication vs.  ips-forms}
\label{sec:pmult-ips}

We begin by examining in some detail the topological structure of $\LDH$ (or, more generally, of a \po\ $\MM$) when it is endowed with the topology ${\sf t}_s$ or ${\sf t}_{s^\ast}$.

As already mentioned, $\LDH$ contains a distinguished *-algebra $\mathfrak{P}$, ${\sf t}_{s^\ast}$-dense. It is easily seen that both the left and right multiplications
by fixed elements of $\mathfrak{P}$ are continuous for the two topologies ${\sf t}_s$ and ${\sf t}_{s^\ast}$.

\berem \label{rem_semiass}The semi-associativity with respect to $\mathfrak{P}$ can be easily checked as follows, without making reference to the topological regularity.
 Let $A_1, A_2\in\LDH$ with $A_1\mult A_2$ well-defined and $B\in\mathfrak{P}$. Then $A_2:\dd\rightarrow D({{A_1}\ad}^*)$ and $A_1:\dd\rightarrow D({{A_2}}^*)$.
Since $B:\dd\rightarrow\dd$, this implies that $A_2\mult B:\dd\rightarrow D({{A_1}\ad}^*)$.  On the other hand, we have that, for $\xi,\eta\in\dd$:
$$
\ip {A_2\mult B\xi}{A_1\ad\eta}=\ip { B\xi}{A_2\ad\mult A_1\ad\eta}=\ip { \xi}{B^*(A_2\ad\mult A_1\ad)\eta};
$$
this implies that $A_1\ad\eta\in D((A_2\mult B)^*)= D((A_2 B)^*)$. In conclusion $A_1\in L^{\w}(A_2\mult B)$. \enrem

Elements of $\mathfrak{P}$ are left ${\sf t}_s$-bounded in the sense of \eqref{ltau-bdd} (see also
\cite{antratsc}); but the set of all left ${\sf t}_s$-bounded elements is larger, namely it is $\LBDH$,
 and it is a C*-algebra.

Another relevant feature of $\LDH$ is the existence of sufficiently many ips-forms. Indeed, if $\xi \in \D$, then every positive sesquilinear form $\vp_\xi$ with
$$\vp_\xi(X,Y):=\ip{X\xi}{Y\xi}$$
 is a ${\sf t}_s$-continuous (and, {\emph{a fortiori}}, ${\sf t}_{s^\ast}$-continuous) ips-form. With the words {\em sufficiently many},
we mean that the unique element $X \in \LDH$ such that \mbox{$\vp_\xi(X,X)=0$,} for every $\xi \in \D$, is   $X=0$ .

The family $\M=\{\vp_\xi;\, \xi \in \D\}$ can also be used to describe the weak multiplication $\mult$ of $\LDH$. Indeed, we have:
\begin{prop} \label{3.2}The weak multiplication $X\mult Y$ of two elements $X, Y \in \LDH$ is well-defined if and only if there exists $Z \in \LDH$ such that
\begin{equation}\label{wp_forms}\vp_\xi (YA,X\ad B)= \vp_\xi (ZA, B), \quad \forall \xi \in \D,\, A,B \in \mathfrak{P}.\end{equation}
\end{prop}
\begin{proof} The necessity of the condition follows easily from \eqref{altwp}. As for the sufficiency, one can put $A=B=I_\D$ in \eqref{wp_forms} and use the polarization identity for getting  \eqref{altwp}.
\end{proof}
Another characterization of the existence of the weak multiplication can be given in terms of approximation by elements of $\mathfrak{P}$.
\begin{prop} \label{3.3}The weak multiplication $X\mult Y$ of two elements $X, Y \in \LDH$ is well-defined if and only if there exists a net $\{B_\alpha\}$
 of elements of  $\mathfrak{P}$ such that
\begin{equation}\label{wp_appr}
B_\alpha \stackrel{{\sf t}_s}{\to} Y \quad \mbox{and}\quad X\mult B_\alpha\; \mbox{converges weakly to some $Z\in \LDH$.}
 \end{equation}
\end{prop}
\begin{proof}Assume that $Y$ satisfies \eqref{wp_appr}. Then we have, for every $\xi, \eta \in \D$,
$$ \ip{Y\xi}{X\ad\eta} = \lim_\alpha \ip{B_\alpha\xi}{X\ad\eta}=\lim_\alpha \ip{X\mult B_\alpha\xi}{\eta} = \ip{Z\xi}{\eta}.$$
The statement then follows from \eqref{altwp}.

On the other hand, assume that $X\mult Y$ is well-defined and let $\{B_\alpha\}$ be a net in $\mathfrak{P}$ converging to $Y$.
Then, for every $\xi, \eta \in \D$,
$$ \lim_\alpha\ip{X\mult B_\alpha\xi}{\eta} = \lim_\alpha \ip{B_\alpha\xi}{X\ad\eta}= \ip{Y\xi}{X\ad\eta}=\ip{X \mult Y\xi}{\eta}.$$
\end{proof}

The strong multiplication of $\LDH$, given by \eqref{eqn_strongmult}, can be conveniently described also by means of the vector forms defined by the inner product of $\H$. To prove this result we need the following lemma.

\begin{lemma}\label{lemma}Let $X \in \LDH$. Then
\begin{itemize}
  \item[(i)] The operator $S(X):=(I +\overline{X}X\ha)^{-1}\upharpoonright \D$ is a weak left multiplier of $X$ and a weak right multiplier of $X\ad$;
  \item[(ii)] $S(X)\D$ is a core for $X\ha$.
\end{itemize}
\end{lemma}
\begin{proof} (i): We need to prove that $X:\D \to D(S(X)\ad\x)$ and $S(X)\ad: \D \to D(X\x)$. The first condition is trivially satisfied since $D(S(X)\ad\x)=\H$, the operator $S(X)$ being symmetric and bounded. For the second, we have
$$S(X)\ad \D= S(X)\D \subset \overline{S(X)} \H = D(\overline{X}X\x)\subset D(X\x).$$

\noi (ii): First, we check that $S(X)\D$ is dense in $\H$. Let $\eta \in \H$ be such that $\ip{S(X)\xi}{\eta}=0$, for every $\xi \in \D$. Then
$$ \ip{\xi}{S(X)\eta}=\ip{S(X)\xi}{\eta}=0, \;\forall \,\xi\in \D.$$
By the density of $\D$, we get $S(X)\eta=0$. But $S(X)$ is one-to-one, thus $\eta=0$.
To prove that $S(X)\D$ is a core for $X^*$, it is enough to show that the unique vector $\{\phi, X\x \phi\}$ in the graph of $X\x$ which is orthogonal to $\{ \, \{\eta, X\x \eta\};\, \eta \in S(X)\D\}$ is zero.
Indeed, putting $\eta= S(X)\xi$, $\xi \in \D$,
\begin{align*}\ip{\{\phi, X\x \phi\}}{\{\eta, X\x \eta\}}&= \ip{\{\phi, X\x \phi\}}{\{S(X)\xi, X\x S(X)\xi\}}\\
&= \ip{\phi}{S(X)\xi} + \ip{X\x \phi}{X\x S(X)\xi}\\
&= \ip{\phi}{S(X)\xi} + \ip{ \phi}{\overline{X}X\x S(X)\xi}\\
&=\ip{\phi}{(I+\overline{X}X\x )S(X)\xi}\\
&= \ip{\phi}{(I+\overline{X}X\x) (I +\overline{X}X\x)^{-1}\xi}\\
&= \ip{\phi}{\xi}=0, \; \forall\, \xi \in \D.
\end{align*}
Hence $\phi=0$. In the previous computation we took into account the following facts: (a) the operator $X\x S(X)$ is bounded; (b) the operator
\mbox{$\overline{X}X\x (I +\overline{X}X\x)^{-1}$} is also everywhere defined and bounded; hence we get 
 \mbox{$X\x (I +\overline{X}X\x)^{-1}\xi \in D(\overline{X})$}, for every $\xi \in \D$,.
\end{proof}

\betheo \label{thm_stronquasistrong}Let $X,Y \in \LDH$. The following statements are equivalent:
\begin{itemize}
  \item[(i)] $X \in L^{\rm s}(Y)$
  \item[(ii)] $X \in L^{\w}(Y)$ and
\begin{itemize}
\item[(ii$_1$)]$\ip{(X\mult Y)\xi}{Z\ad \eta}=\ip{Y\xi}{(X\ad\mult Z\ad)\eta}, \; \forall\, Z \in L^{\w}(X), \xi, \eta \in \D;$
\item[(ii$_2$)]$\ip{(Y\ad\mult X\ad)\xi}{V \eta}=\ip{X\ad\xi}{(Y\mult V)\eta}, \; \forall\, V \in R^{\w}(Y), \xi, \eta \in \D$.
\end{itemize}
\end{itemize}
\entheo

\begin{proof}

Let $X,Y \in \LDH$.
The implication (i)$\Rightarrow$(ii) is easy. We prove that (ii)$\Rightarrow$(i).
 Let $X,Y \in \LDH$ satisfy (ii).  We begin with observing that the conditions (ii$_1$) and (ii$_2$) are, respectively, equivalent to the following ones\footnote
 {\,We remind the reader that if $T$ is not densely defined, $D(T\ha)=\{\eta \in \H: \exists \eta\ha\in \H: \, \ip{T\xi}{\eta}=\ip{\xi}{\eta\ha} ,\, \forall \xi \in D(T)\}$ is not necessarily the domain of a well-defined operator.}
\begin{align*}
 &Y:\D \to D((X\ha\up Z\ad \D)\ha), \; \forall \,Z \in L^{\w}(X) \\ \mbox{ and }
 &X\ad:\D \to D(({Y}\ad\x \up V \D)\ha), \; \forall \,V \in R^{\w}(Y).
\end{align*}
By Lemma \ref{lemma}, $S(X) \in L^{\w}(X)$ and since $S(X)\D$ is a core for $X\x$,
$$ (X\ha\upharpoonright S(X) \D)\ha= (X\x)\x = \overline{X}.$$
Thus, $Y:\D \to D(\overline{X})$.
By applying again Lemma \ref{lemma} to the operator $Y\ad$ we obtain that $S(Y\ad)$ is a right multiplier of $Y$ and $S(Y\ad)\D$ is a core for $Y\ad\x$. Then, $X\ad:\D\to D(\overline{Y\ad})$.
In conclusion $Y \in L^{\rm s}(X)$. \end{proof}

An interesting aspect of the interplay of the weak and strong multiplication in $\LDH$ is the following  \emph{mixed}  associativity property
 \cite[Prop.3.5]{antratsc}, which proves to be useful in many situations.
\begin{prop} Let $X,Y, Z \in \LDH$. Assume that $X\mult Y$, $(X\mult Y)\mult
Z$ and $Y \circ Z$ are all well-defined. Then $X \in L^{\rm w}(Y \circ Z)$ and
\begin{equation}\label{assoc3}X\mult(Y\circ Z)= (X\mult Y) \mult Z.\end{equation}
\end{prop}
\noi In other words, \eqref{assoc} is valid in $\LDH$ with the strong partial multiplication.

\berem \label{rem_topreg} The partial O*-algebra $\LDH$ is topologically regular when endowed with the strong topology ${\sf t}_s$. Indeed, the multiplication induced by ${\mathfrak P}$ is a restriction of the strong multiplication of $\LDH$, since, if $Y\in D(\ZOL_X)$,
then there exists a net $\{Y_\alpha\} \subset \mathfrak{P}$ and $Z \in \LDH$  such that $Y_\alpha \xi \to Y\xi$ and $X\mult Y_\alpha \xi \to Z\xi$, for every $\xi \in \D$. This implies that $Y\xi \in D(\overline{X})$ and $Z\xi = \overline{X}Y\xi$, for every $\xi \in \D$. In a similar way, one proves that, if $X\ad \in D(\ZOL_{Y\ad})$,  then  $X\ad\xi \in D(\overline{Y\ad})$. Hence, if the multiplication induced by ${\mathfrak P}$ of $X$ and $Y$ is well-defined, then $X\circ Y$ is also well-defined and the two products coincide. The statement then follows from \eqref{assoc3}.

\enrem

The next two statements are the analogues of Proposition \ref{3.2} and Proposition \ref{3.3} and can be proved in a similar way.
\begin{prop} The strong multiplication $X\circ Y$ of two elements $X, Y \in \LDH$ is well-defined if and only if there exists $W \in \LDH$ such that
\begin{align*}\vp_\xi (WA,Z\ad B)&= \vp_\xi (YA, (X\ad\mult Z\ad) B),\\
&  \mbox{ whenever }Z\in L^{\w}(X),\; \forall \,\xi \in \D,\, A,B \in \mathfrak{P},
\end{align*}
and
\begin{align*}\vp_\xi (W\ad A,V B)&= \vp_\xi (X\ad A, (Y\mult V) B),\\
&  \mbox{ whenever }V\in R^{\w}(Y),\;\forall \,\xi \in \D,\, A,B \in \mathfrak{P}.
\end{align*}
\end{prop}

{\begin{prop} The strong multiplication $X\circ Y$ of two elements $X, Y \in \LDH$ is well-defined if and only if there exist $W \in \LDH$ and a net $\{C_\alpha\}$
of elements of  $\mathfrak{P}$ such that $C_\alpha \stackrel{{\sf t}_s}{\to} Y$ and
$$
 \vp_\xi((X\mult C_\alpha -W)A, Z\ad B)\to 0,\; \mbox{ if } Z\in L^{\w}(X), \forall \,\xi \in \D,\, A,B \in \mathfrak{P},
$$
$$
 \vp_\xi((C_\alpha\ad\mult X\ad -W\ad)A, V B)\to 0,\; \mbox{ if } V\in R^{\w}(Y), \forall \,\xi \in \D,\, A,B \in \mathfrak{P}.
$$
\end{prop}
}

\bigskip The family $\M=\{\vp_\xi;\, \xi \in \D\}$ plays  an important role in the preceding discussion. Even though the elements of $\M$ do not exhaust the family of all strongly continuous ips-forms on $\LDH$, it is not restrictive to confine the analysis to them, since every ${\sf t}_s$-continuous ips-form on $\LDH$ is a linear combination of elements of $\M$. Indeed:

\betheo Let $\MM$ be a partial O*-algebra on $\D$, $\MM_0$ a *-algebra of bounded operators contained in $\MM$ and strongly* dense in $\MM$.
\begin{itemize}
\item[(i)] Then every strongly continuous invariant positive sesquilinear form $\vp$ on $\D \times \D$, with core  $\MM_0$, can be represented as
\begin{equation}\label{eq_repr} \vp(X,Y)= \sum_{i=1}^n \ip{S_iX\xi_i}{S_iY\xi_i}, \quad X,Y \in \MM, \end{equation}
 for some vectors $\xi_1, \ldots, \xi_n$ of $\D$ and positive operators $S_1, \ldots, S_n$ such that $S_1^2, \ldots, S_n^2 \in \MM'_{{\rm q}\w}$.

\item[(ii)] If $\MM=\LDH$ and $\MM_0= \mathfrak{P}$, then
$$ \vp(X,Y)= \sum_{i=1}^n \ip{X\xi_i}{Y\xi_i}, \quad X,Y \in \MM,
$$
for some vectors $\xi_1, \ldots, \xi_n$ of $\D$.
 \end{itemize}
\entheo
\begin{proof}
The strong continuity of $\vp$ implies that there exists vectors $\xi_1, \ldots, \xi_n\in \D$ such that
$$ |\vp(X,Y)| \leq \sum_{i=1}^n p_{\xi_i}(X)\cdot \sum_{i=1}^n p_{\xi_i}(Y).$$

Let $\H\ssp:= \oplus_{i=1}^n \H$, the direct sum of $n$ copies of $\H$. We will write $\oplus\xi_i$ instead of $(\xi_1, \ldots, \xi_n)$, $\xi_i \in \H$.
Let $\D\ssp= \oplus_{i=1}^n\D$.\\

We define a *-representation $\pi$ of $\MM$ into $\LDHO{\D\ssp}{\H\ssp}$ by
$$
\pi(X) (\oplus\eta_i) = \oplus X\eta_i, \quad \eta_i \in \D, i=1,\ldots, n.
$$
Let us consider the following subspaces of $\oplus_{i=1}^n \H$:
$$
\E= \{ \pi(X)(\oplus\xi_i); X \in \MM\}, \quad \E_0= \{ \pi(A)(\oplus\xi_i); A \in \MM_0\}.
$$
The strong *-density of $\MM_0$ implies that $\overline{\E_0}= \overline{\E}$.

Define
$$
\Theta (\pi(X)(\oplus\xi_i), \pi(Y)(\oplus\xi_i)):= \vp(X,Y).
$$
The sesquilinear form $\Theta $ is bounded on $\E\times \E$ and extends to $\overline{\E}\times \overline{\E}$.
Then there exists a positive bounded operator $T$ in the Hilbert space $\overline{\E}$ such that
$$
 \Theta (\pi(X)(\oplus\xi_i), \pi(Y)(\oplus\xi_i))= \ip{T(\oplus X\xi_i)}{\oplus Y\xi_i}.
$$

The condition $\vp(X\mult A, B)= \vp(A, X\ad \mult B)$ implies the equality:

$$\ip{T\pi(X\mult A)\oplus\xi_i}{\pi(B)\oplus\xi_i}=\ip{T\pi(A)\oplus\xi_i}{\pi(X\ad\mult B)\oplus\xi_i},$$
or,
\begin{equation}\label{comm} \ip{T(\pi(X)\mult\pi( A))\oplus\xi_i}{\pi(B)\oplus\xi_i}=\ip{T\pi(A)\oplus\xi_i}{(\pi(X\ad)\mult\pi( B))\oplus\xi_i}.\end{equation}
Now, for every $X \in \MM$, we define an operator $\pi_\E$ on $\E_0$ by
$$ \pi_\E(X) (\pi(A)\oplus\xi_i):= (\pi(X)\mult \pi(A))\oplus\xi_i, \quad A \in \MM_0.$$
It is easily seen that $\pi_\E(X) \in \LDHO{\E_0}{\overline{\E}}$. With this notation, \eqref{comm} reads as follows
$$ \ip{T\pi_\E(X)(\pi( A)\oplus\xi_i)}{\pi(B)\oplus\xi_i}=\ip{T\pi(A)\oplus\xi_i}{(\pi_\E(X\ad)(\pi( B)\oplus\xi_i}.$$
Hence $T\in {\pi_\E(\MM)}'_{\w}$.

Now we extend $T$ to a bounded operator $T\ssp$ on $\H\ssp$ by putting it to be $0$ on the orthogonal complement of $\overline{\E}$.

Now we prove that $T\ssp\in \pi(\MM)'_{\w}$. Recalling that $\pi(\MM_0)$ is a *-algebra of bounded operators, we begin with showing that $T\ssp \in \overline{\pi(\MM_0)}'$ (the ordinary commutant of bounded operators). Let $P_\E$ denote the projection of $\H\ssp$ onto $\overline{\E}$. Since every $\overline{\pi(A)}$, $A \in \MM_0$, leaves $\overline{\E}$ invariant, it follows $\overline{\pi(A)}P_\E=P_\E\overline{\pi(A)}$, for every $A\in \MM_0$. Moreover, if $\oplus\eta_i\in \H\ssp$, there exists a sequence $\{B_n\}$ of elements of $\MM_0$ such that $P_\E\oplus\eta_i =\lim_{n\to\infty} \pi(B_n)\oplus\xi_i$. From these facts, we get
\begin{eqnarray*}T\ssp \overline{\pi(A)}P_\E \oplus\eta_i &=& T\ssp \overline{\pi(A)}( \lim_{n\to\infty} \pi(B_n)\oplus\xi_i)\\ &=& \lim_{n\to\infty}T\ssp \overline{\pi(A)}\pi(B_n)\oplus\xi_i\\
&=& \lim_{n\to\infty}\overline{\pi(A)}T\ssp \pi(B_n)\oplus\xi_i \\
&=& \overline{\pi(A)}T\ssp P_\E \oplus\eta_i.
\end{eqnarray*}
Moreover, by the definition of $T\ssp$ we obtain $T\ssp \overline{\pi(A)}(I-P_\E) \oplus\eta_i= T\ssp(I-P_\E)\overline{\pi(A)} \oplus\eta_i=0$ and thus $T\ssp \in \overline{\pi(\MM_0)}'$.
Since $\overline{\MM_0}^{s^*} \supseteq \MM$, it follows that $\pi(\MM)'_{\w} = \pi(\MM_0)'_{\w} = \overline{\pi(\MM_0)}'$ and we finally conclude that $T\ssp\in \pi(\MM)'_{\w}$.

On the other hand, the condition $\vp(X\ad\mult A, Y B)= \vp(A, (X\mult Y) \mult B)$, whenever $X\mult Y$ is defined, implies, in similar way, that $T\ssp\in \pi(\MM)'_{{\rm q}\w}$.
 Let now $T_i$ denote the projection of $T$ onto the subspace generated by $X\xi_i$, $X \in \MM$ and then extended to $\H$ by defining it as $0$ on the orthogonal complement.
It is easily seen that $T\ssp \in \pi(\MM)'_{{\rm q}\w}$ if, and only if $T_i \in \MM'_{{\rm q}\w}$ for each $i$.

Hence,
$$
\vp(X,Y)= \sum_{i=1}^n{ \ip{T_i X\xi_i}{Y\xi_i}},\qquad \xi_i \in \D, T_i \in \MM'_{{\rm q}\w}\,.
$$
If we put $S_i=T_i^{1/2}$, then we get the representation \eqref{eq_repr}.
If $\MM=\LDH$, then (ii) follows from the equality $\LDH'_{{\rm q}\w}= {\mb C}I$.
\end{proof}

With a similar proof, one also gets
\betheo Let $\MM$ be a partial O*-algebra on $\D$. Every strongly continuous linear functional $\Phi$ can be represented as
$$ \Phi (X) = \sum_{i=1}^n \ip{X\xi_i}{\eta_i}, \quad X \in \MM$$
with $\xi_1, \ldots \xi_n \in \D$ and $\eta_1, \ldots \eta_n \in \H$.
\entheo

In \cite{antratsc} we gave the following definition of a partial GC*-algebra of operators.

\bedefi A partial O*-algebra $\MM$ on $\D$ is called a \emph{partial GC*-algebra of operators} over $\MM_0$
if\begin{itemize}
    \item[(i)] $\MM$ is ${\sf t}_{s^\ast}$-closed;
    \item[(ii)] $\MM$ contains a ${\sf t}_{s^\ast}$-dense *-algebra $\MM_0$ of bounded operators on $\D$;
    \item[(iii)]$\MM_{\sf lb}= \MM \cap \LBDH=:\MM_b$ is a
    C*-algebra.
    \end{itemize}
\findefi
\berem \label{rem_311b} Every partial GC*-algebra of operators is topologically regular. Indeed, the argument used in Remark \ref{rem_topreg} can be easily adapted to the present situation. Hence, every partial GC*-algebra of operators is a partial GC*-algebra in the sense of Section \ref{sect_preliminaries}.

\enrem
\medskip

Clearly $\LDH$ fulfills this definition if $\MM_0=\mathfrak{P}$. So it is natural to consider under which conditions
a locally convex partial *-algebra $\A[\tau]$ can be represented into a partial GC*-algebra of operators. Some results in this direction were given in \cite{antratsc}, but a deeper analysis shows that the conditions given there were sometimes unnecessarily strong. The crucial point for the existence of a nice *-representation of $\A[\tau]$ is that it possesses a sufficient family of ips-forms as $\LDH$ itself does. This will be the starting point of the present discussion.

\section{Sufficient families of ips-forms, $\M$-bounded elements}
\label{sec:Mbdd}

\bedefi Let $\A$ be a  partial *-algebra endowed with a locally convex topology $\tau$, generated by a directed set of seminorms
$\{p_\alpha\}_{\alpha \in I}$.
We say that $\A[\tau]$ is \emph{$\Ao$-regular} if there exists a *-algebra $\Ao \subset R\A$  with the following properties:
\begin{itemize}
\item[({\sf d}$_1$)]
$\Ao$ is $\tau$-dense in $\A$;
\item[({\sf d}$_2$)]
for every $b\in\Ao$, the maps $x\mapsto xb$ and $x\mapsto bx, \, x\in \A$, \mbox{are continuous.}
\end{itemize}
\findefi
\berem We warn the reader that an $\Ao$-regular partial *-algebra $\A[\tau]$ is not necessarily a locally convex partial *-algebra in the sense of \cite{ait_book}, since the definition of the latter requires stronger conditions (for instance, the continuity of the involution and of the multiplication $x \mapsto xb$ for every fixed $b\in R\A$).
\enrem

Let now $\mathcal M$ be a family of positive sesquilinear forms on $\A\times\A$ for which the conditions ({\sf ips}$_1$), ({\sf ips}$_3$) and ({\sf ips}$_4$) are satisfied with respect to $\Ao$ and such that every
$\varphi \in {\mc M} $ is $\tau$-continuous, i.e., there exists $p_\alpha$, $\gamma>0$ such that:
$$
|\varphi(x,y)|\leq \gamma \,p_\alpha(x)p_\alpha(y).
$$
Then ({\sf ips}$_2$) is also satisfied and, therefore,
$\Ao$ is a core for every $\varphi \in {\mc M}$, so that every $\varphi\in {\mc M}$ is an ips-form.

As announced above, the crucial condition is that  $\A$ possesses  sufficiently many  ips-forms. Hence, as in \cite{antratsc}, we  introduce
\bedefi
A   family $\mathcal M$ of ips-forms on $\A\times\A$ with the above properties is {\it sufficient}
 if $x\in \A$ and $\vp(x,x)=0$ for every $\vp \in {\mc M}$ imply $x=0$.
\findefi
This condition is not empty, as the following examples show.
Take $L^p[0,1]$: for $1\leq p<2$ the family is trivial, so is not sufficient. For $p\geq
2$, the family is sufficient. In the example $L^p[0,1]\oplus L^r[0,1]$ for $1\leq p< 2$ and $r\geq 2$, the family of forms is neither
sufficient, nor trivial.

Of course, if the family $\M$ is sufficient, any larger family $\M' \supset \M$ is also sufficient. The maximal sufficient family is obviously the set $\P_{\Ao}(\A)$ of \emph{all} continuous ips-forms with core $\Ao$, but we prefer to use the present notion, since it provides more flexibility.

When $\A$ possesses a sufficient family  $\mathcal M$ of ips-forms, we can
define an {\em extension} of the multiplication in the following way.

We say that the {\it weak} multiplication $x\wmult y$ is well-defined if there exists $z\in\A$ such that:
$$
\varphi(ya,x^*b)=\varphi(za,b),\;\forall\, a,b\in\Ao, \forall\,\varphi\in\mathcal M.
$$
In this case, we put $x\wmult y:=z$.

The following result is immediate.
\begin{prop}
If the partial *-algebra $\A$  possesses a sufficient family  $\mathcal M$ of ips-forms, then
$\A$  is also a  partial *-algebra with respect to the weak multiplication.
\end{prop}

\medskip
From now on we will consider only the case where $\A$ possesses a sufficient family $\M$ of ips-forms.

\medskip

\berem
The sesquilinear forms of $\mathcal M$ define the  topologies   generated by the following families of seminorms:
\begin{itemize}
\item[$\tau_w^{\scriptscriptstyle\M} $:]
$\quad x\mapsto |\varphi(x a,b)|$, $\quad\varphi\in\mathcal M, a,b\in\Ao$;
\item[$\tau_s^{\scriptscriptstyle\M} $:]
$\quad x\mapsto \varphi(x,x)^{1/2}$, $\quad\varphi\in\mathcal M$;
\item[$\tau_{s^*}^{\scriptscriptstyle\M} $:]
$\quad x\mapsto \max\{\varphi(x,x)^{1/2},\varphi(x^*,x^*)^{1/2}\}, \quad\varphi\in\mathcal M$
\end{itemize}
From the continuity of $\varphi\in\mathcal M$ it follows that all
the topologies $\tau^{\scriptscriptstyle\M} _w$, $\tau^{\scriptscriptstyle\M} _s$, (and also $\tau^{\scriptscriptstyle\M} _{s^*}$, if the
involution is $\tau$-continuous) are coarser than the initial
topology $\tau$. \enrem

We have the following
\begin{prop}\label{prop_43}
The weak product $x\wmult y$ is defined if, and only if, there exists a net $\{b_\alpha\}$ in $\Ao$ such that $b_\alpha\stackrel{\tau}\longrightarrow y$
and $x b_\alpha\stackrel{\tau^{\scriptscriptstyle\M} _w}\longrightarrow z\in\A$
\end{prop}
\begin{proof}
Assume that $x\wmult y$ is defined. From the $\tau$-density of $\Ao$, there exists a net $\{b_\alpha\}$ in $\Ao$ such that $b_\alpha\stackrel{\tau}\longrightarrow y$.
 Then one has, for every $c,c'\in \Ao$: $\varphi((xb_\alpha)c,c')=\varphi(b_\alpha c,x^*c')\rightarrow\varphi(yc,x^*c')=\varphi((x\mult y) c,c')$, that is,
 $x b_\alpha\stackrel{\tau^{\scriptscriptstyle\M} _w}\longrightarrow x\mult y$. Conversely, assume that there exists a net $\{b_\alpha\}$ in $\Ao$ such that
$b_\alpha\stackrel{\tau}\longrightarrow y$ and $x b_\alpha\stackrel{\tau^{\scriptscriptstyle\M} _w}\longrightarrow z\in\A$.
 Then, for every $a,a'\in \Ao$ $\varphi(ya,x^*a')=\lim_\alpha\varphi(b_\alpha a,x^*a')=\lim_\alpha\varphi((xb_\alpha)a,a')=\varphi(za,a')$, that is, $x\wmult y$ is defined.
\end{proof}

In the case of $\LDH$, the weak multiplication $\mult$ coincides with the weak multiplication defined here by means of ips-forms (Proposition \ref{3.2}). By analogy, from now on we will always suppose that the following condition holds:

\begin{itemize}
\item[{\sf (wp)}] $xy$ exists if, and only if, $x\mult y$ exists. In this case $xy=x\mult y$.
\end{itemize}
Then, of course, $L(x) = L^{\w}(x)$ and  $R(x) = R^{\w}(x)$.
\medskip

The first result is that, if $\A$ is a \pa\ with a sufficient family $\M$ of ips-forms, and satisfying {\sf (wp)}, then,
it satisfies the condition {\sf (cl)} with respect to the topology  $\tau^{\scriptscriptstyle\M}_s$.

\begin{prop}\label{prop_415}
Let $\A$ be a \pa\ with a sufficient family $\M$ of ips-forms, and satisfying {\sf (wp)}. Then,
for every $x \in \A$, the linear map $\OL_x: R(x)\to \A$ with
$\OL_x(y)=xy$, $y\in R(x)$ is closed with respect to $\tau^{\scriptscriptstyle\M} _s$, in
the sense that if $y_\alpha \stackrel{\tau^{\scriptscriptstyle\M} _s}{\to} y$, with $y_\alpha\in R(x)$
and $xy_\alpha \stackrel{\tau^{\scriptscriptstyle\M} _s}{\to} z \in \A$, then $y\in R(x)$ and
$z=xy$.
\end{prop}
\begin{proof} Let $y_\alpha \stackrel{\tau^{\scriptscriptstyle\M} _s}{\to} y$, with $y_\alpha\in R(x)$
and $xy_\alpha \stackrel{\tau^{\scriptscriptstyle\M} _s}{\to} z \in \A$. Then, again by ({\sf ips}$_4$), for every $\vp\in\M$,
\begin{align*}\vp((xy_\alpha-z)a,a')&= \vp((xy_\alpha)a,a')-\vp(za,a')
\\
&=\vp(y_\alpha a,x\ha a')-\vp(za,a')\to \vp(ya,x\ha a')-\vp(za,a') =0.
\end{align*}
Hence, since $\M$ is sufficient, $y\in R(x)$ and $z=xy$.
\end{proof}

\berem It is clear that the same statement of Proposition \ref{prop_415} holds for any topology finer than $\tau^{\scriptscriptstyle\M} _s$ and then, in particular, for the initial topology $\tau$ of $\A$.
\enrem

Now we are ready to introduce the appropriate notion of bounded elements.
\bedefi
Let $\A$ be a \pa\ with a sufficient family $\M$ of ips-forms, and satisfying {\sf (wp)}.
An element $x\in\A$ is called \emph{$\mc M$-bounded} if there exists $\gamma>0$ such that:
$$
|\vp(xa,b)|\leq\gamma\, \vp(a,a)^{1/2}\varphi(b,b)^{1/2}, \;\forall\, \vp\in\mathcal M,\, a,b\in\Ao\,.
$$
\findefi
\begin{prop}
Let $\A[\tau]$ be an $\Ao$-regular \pa\ satisfying condition {\sf (wp)}. Then, an element
 $x\in\A$ is $\M$-bounded if, and only if there exists $\gamma\in\mathbb R$ such that $\varphi(xa,xa)\leq\gamma^2\varphi(a,a)$ for all $\varphi\in \M$ and $a\in\Ao$.\
\label{boundedprop}
\end{prop}
\begin{proof}
Assume that $x\in\A$ is $\M$-bounded. By the density of $\Ao$, there exists a net $\{x_\alpha\}\in\Ao$ such that $\tau-\lim_\alpha x_\alpha=x$.
The continuity of $\varphi$ then implies:
\begin{align*}
|\varphi(xa,xb)|&=\lim_\alpha |\varphi(xa,x_\alpha b)|\leq\gamma \,\varphi(a,a)^{1/2}\lim_\alpha \varphi(x_\alpha b,x_\alpha b)^{1/2}
\\ &=\gamma\,\varphi(a,a)^{1/2} \varphi(xb,xb)^{1/2}.
\end{align*}
In particular, it follows that
$$
\varphi(xa,xa)\leq\gamma\,\varphi(a,a)^{1/2} \varphi(xa,xa)^{1/2},
$$
that is $\varphi(xa,xa)\leq\gamma^2\varphi(a,a).$

Conversely, we have:
$$
|\varphi(xa,b)|\leq\varphi(xa,xa)^{1/2}\varphi(b,b)^{1/2}\leq\gamma\,\varphi(a,a)^{1/2}\varphi(b,b)^{1/2}
$$
\end{proof}
From the last proposition, it follows obviously that an element $x$
of $\A$ is $\mathcal M$-bounded if, and only if, $x$ is left
$\tau^{\scriptscriptstyle\M} _s$-bounded, in the sense of \cite{antratsc}. Let us  define:
\begin{align*}
q_{\scriptscriptstyle\M}(x)&:=\inf\{\gamma>0: \varphi(xa,xa)\leq\gamma^2\varphi(a,a), \,\forall\,\varphi\in\mathcal M, \,\forall\,  a\in\Ao\}
\\
&= \sup\{\varphi(xa,xa)^{1/2}: \varphi\in \mc M, \,a\in
\Ao,\, \varphi(a,a)^{1/2}=1\}
\end{align*}
Hence $q_{\scriptscriptstyle\M}$ coincides with the norm $\nlb{\cdot}$ obtained by giving $\A$ the topology $\tau_s^{\scriptscriptstyle\M} $ (see also \cite{cimpric}for a similar approach)).

Then the following holds:
\begin{prop}\label{prop_algebra} Let $x,y$ be $\M$-bounded elements of $\A$. The following statements hold:
\begin{itemize}
  \item[(i)] $x^*$ is ${\mc M}$-bounded also, and $q_{\scriptscriptstyle\M}(x)=q_{\scriptscriptstyle\M}(x^*)$;
  \item[(ii)] If $xy$ is well-defined, then $xy$ is ${\mc M}$-bounded and
  $$ q_{\scriptscriptstyle\M}(xy)\leq q_{\scriptscriptstyle\M} (x)\, q_{\scriptscriptstyle\M}(y).$$
\end{itemize}
\end{prop}
\begin{proof} (i):
The first part is a direct consequence of the definition. The second part follows from the fact that $|\varphi(xa,b)|=|\varphi(a,x^*b)|=|\varphi(x^*b,a)|$,
by Proposition \ref{boundedprop} and by definition of $q_{\scriptscriptstyle\M}(x)$.
\begin{align*}
|\varphi((xy)a,b)|&=|\varphi(ya,x^*b)|\leq\varphi(ya,ya)^{1/2}\varphi(x^*b,x^*b)^{1/2}\\
&\leq q_{\scriptscriptstyle\M}(x)q_{\scriptscriptstyle\M}(y)\varphi(a,a)^{1/2}\gamma_2\varphi(b,b)^{1/2}.
\end{align*}
Taking the sup on the l.h.s., we get the desired inequality.
\end{proof}
\begin{prop} $q_{\scriptscriptstyle\M}$ is an unbounded $C^*$-norm on $\A$ with domain $\D(q_{\scriptscriptstyle\M}):=\{x\in\A: x~\mbox{is} ~{\mc M}\mbox{-bounded}\}$.
\end{prop}
\begin{proof} This can be deduced from \cite{tratsc} or, simply, computed directly.
\end{proof}

The existence of a  sufficient family $\M$ of ips-forms allows the definition of a stronger multiplication on $\A$, that will play a role similar to the strong partial multiplication on $\LDH$.
\bedefi\label{strongprod2} If the family $\M$ of ips-forms is
sufficient, we say that the {\it strong} multiplication $x\bullet y$
is well-defined (and that $x\in L^s(y)$ or $y\in R^s(x)$) if $x\in L(y)$ and:
\begin{itemize}
\item[({\sf sm}$_1$)]$\varphi((xy)a,z^*b)=\varphi(ya,(x^*z^*)b), \; \forall \,z \in L(x),\forall\,\varphi\in{\M}, \forall\, a,b\in\A_0$;
\item[({\sf sm}$_2$)]$\varphi((y^*x^*)a,vb)=\varphi(x^*a,(yv)b), \; \forall\, v \in R(y),\forall\, \varphi\in{\M},\forall\, a,b\in\A_0.$
\end{itemize}
\findefi
The following characterization is immediate.
\begin{prop} \label{strongprod} If the family $\M$ of ips-forms is sufficient, the {\it strong} multiplication $x\bullet y$
is well-defined (and $x\in L^s(y)$ or $y\in R^s(x)$) if, and only if, there exists $w\in\A$ such that:
$$
\varphi(wa,z^*b)=\varphi(ya,(x^*z^*)b)\;\mbox{ whenever }\; z\in L(x),\;\forall\, \varphi\in\mathcal M, \,\forall\, a,b\in\Ao,
$$
and
$$
\varphi(w^*a,vb)=\varphi(x^*a,(yv)b)\;\mbox{ whenever }\; v\in R(y),\;\forall\, \varphi\in\mathcal M,\,\forall \,a,b\in\Ao.
$$
In this case, we put $x\bullet y:=w$. \end{prop}
 \berem
 The uniqueness of $w$ results  from the sufficiency of the family $\mathcal M$.
Clearly, if $\A$ has a unit, then $x\bullet y:=w$ implies that $x y$ is defined and $x\bullet y=x y=w$.
\enrem
We have the following
\begin{prop}
The strong multiplication $x\bullet y$ of two elements $x, y \in \A$
is well-defined if and only if there exist $w \in \A$ and a net
$\{c_\alpha\}$ of elements of  $\Ao$ such that $c_\alpha
\stackrel{\tau^{\scriptscriptstyle\M} _s}{\to} y$ and
$$
 \vp((x\mult c_\alpha -w)a, z^* b)\to 0,\; \mbox{ if } z\in L(x), \forall \,\vp \in \mathcal M,\, a,b \in \Ao,
$$
$$
 \vp((c_\alpha^*\mult x^* -w^*)a, v b)\to 0, \mbox{ if } v\in R(y), \forall \, \vp \in \mathcal M,\, a,b \in \Ao.
$$
\end{prop}
\begin{proof} If $x\bullet y$ is well-defined, then $xy$ is well-defined. Then, by Proposition \ref{prop_43}, there exists a net $\{c_\alpha\} \subset \Ao$ such that $c_\alpha
\stackrel{\tau^{\scriptscriptstyle\M} _s}{\to} y$ and $xc_\alpha \stackrel{\tau^{\scriptscriptstyle\M} _w}{\to} xy$. Hence, by ({\sf ips}$_4$) and by the continuity of every $\vp \in \M$,
$$\vp((xc_\alpha-xy)a,a')= \vp(x(c_\alpha-y)a,a')=\vp((c_\alpha-y)a,x\ha a')\to 0$$
and
$$\vp((c_\alpha\ha x\ha - y\ha  x\ha)a,a')= \vp(x\ha a, (c_\alpha-y)a')\to 0.$$
The converse is straightforward.
\end{proof}

\begin{prop} \label{prop_412} Let $x,y$ be $\M$-bounded elements of $\A$. Then $x\bullet y$ is well-defined if and only if $xy$ is well-defined.
\end{prop}

\begin{proof} If $x\bullet y$ is well-defined, then $xy$ is obviously well-defined. Assume that $xy$ is well-defined. Then by Proposition
\ref{prop_algebra}, $xy$ is bounded. Let $z \in L(x)$. For $\vp \in \M$ we denote by $\pi_\vp$ the corresponding GNS representation. Then, as it is easily seen, for every $\M$-bounded element $z$, $\pi_\vp(z)$ is a bounded operator. Hence, for every $a,b\in \Ao$ and $\vp\in \M$,
\begin{align*}
|\vp( (xy)a, z\ha b)| &= |\ip{\pi_\vp(xy)\lambda_\vp(a)}{\pi_\vp(z\ha )\lambda_\vp(b)}|\\
&=  |\ip{\pi_\vp(x) \mult \pi_\vp(y)\lambda_\vp(a)}{\pi_\vp(z\ha )\lambda_\vp(b)}|\\
&\leq \|\overline{\pi_\vp(x)}\| \|\pi_\vp(y)\lambda_\vp(a) \|\, \|\pi_\vp(z\ha)\lambda_\vp(b) \|.
\end{align*}
This implies that $\pi_\vp(z\ha)\lambda_\vp(b) \in D(\pi_\vp(x)\ha)$. Hence,
\begin{align*}
\vp( (xy)a, z\ha b) &= \ip{\pi_\vp(y)\lambda_\vp(a)}{\pi_\vp(x)\ha\pi_\vp(z\ha )\lambda_\vp(b)}\\
&= \ip{\pi_\vp(y)\lambda_\vp(a)}{\pi_\vp(x\ha)\mult\pi_\vp(z\ha )\lambda_\vp(b)}\\
&= \ip{\pi_\vp(y)\lambda_\vp(a)}{\pi_\vp(x\ha z\ha )\lambda_\vp(b)}\\
&= \vp(ya, (x\ha z\ha) b).
\end{align*}
 Condition ({\sf sm}$_2$) is proved in a similar way.
\end{proof}

\begin{prop}\label{prop_413} Let $x,y$ be $\M$-bounded elements of $\A$. Then, for every $\vp\in \M$, $\pi_\vp(x)\mult\pi_\vp(y)$ is well-defined.
\end{prop}
\begin{proof} Indeed, we have, for every $a,b \in \Ao$,
\begin{align*}
|\ip{\pi_\vp(y)\lambda_\vp(a)}{\pi_\vp(x\ha )\lambda_\vp(b)}| &= | \vp(ya, x\ha b)\\
& \leq  \vp(ya,ya)^{1/2} \vp(x\ha b, x\ha b)^{1/2} \\
&\leq  q_{\scriptscriptstyle\M} (x) q_{\scriptscriptstyle\M}(y) \vp(a,a)^{1/2} \vp(b,b) ^{1/2}.
\end{align*}
Then, by the representation theorem for bounded sesquilinear forms in Hilbert space, there exists $Z_\vp \in \B(\H_\vp)$ such that
$$ \ip{\pi_\vp(y)\lambda_\vp(a)}{\pi_\vp(x\ha )\lambda_\vp(b)} = \ip{Z_\vp \lambda_\vp(a)}{\lambda_\vp(b)}.$$
This implies that $\pi_\vp(x)\mult\pi_\vp(y)$ is well-defined.
\end{proof}
\berem We emphasize that this does \emph{not} imply that there exists $z \in \A$ such that \mbox{$\pi_\vp(x)\mult\pi_\vp(y)=\pi_\vp(z)$.}
This fact will motivate a further restriction on the family $\M$, see Definition \ref{def:wbehaved} below.
\enrem
It is natural to ask  under which assumptions $\Ao$ itself consists of bounded elements.

\begin{prop}Let $\A[\tau]$ be $\Ao$-regular. Assume that the directed family $\{p_\alpha\}_{\alpha \in I}$ defining the topology $\tau$ has the property that, for every $\alpha\in I$,
\begin{equation}\label{kapla}
\liminf_{n \to \infty} \left(  p_\alpha((a\ha a)^{2^n})\right)^{2^{-n}}<\infty, \quad \forall a\in \Ao.
\end{equation}
Then every $a\in \Ao$ is $\M$-bounded.
\end{prop}
\begin{proof}Let $a,b \in \Ao$. By the Cauchy-Schwarz inequality, we have
$$
 \vp(ab,ab) =  \vp(b,a\ha a b)  \leq \vp(b,b)^{1/2} \vp(a\ha a b, a\ha a b)^{1/2} =\vp(b,b)^{1/2} \vp(b, (a\ha a)^{2}b)^{1/2}.
$$
Iterating,  one obtains first
$$
 \vp(ab,ab) \leq  \vp(b,b)^{1/2+1/4}\vp((a\ha a)^2 b, (a\ha a)^2 b)^{1/4},
$$
and then the following Kaplansky-like inequality:
$$
 \vp(ab,ab) \leq \vp(b,b)^{1-2^{-(n+1)}} \vp( (a\ha a)^{2^n}b, (a\ha a)^{2^n}b)^{2^{-(n+1)}}.
 $$
By the continuity of $\vp$ and of the right multiplication by $b\in \Ao$, we can find a continuous seminorm $p$ such that
$$
 \vp(ab,ab) \leq \vp(b,b)^{1-2^{-(n+1)}} \left( p((a\ha a)^{2^n})\right)^{2^{-n}} p(b) ^{2^{-n}}.
 $$
On the other hand, there exists $\alpha$ and $\gamma >0$ such that $p(x)\leq \gamma\, p_\alpha(x)$, for every $x \in \A$. Hence,
$$
 \vp(ab,ab) \leq \vp(b,b)^{1-2^{-(n+1)}} \gamma ^{2^{-n+1}} \left( p_\alpha((a\ha a)^{2^n})\right)^{2^{-n}} p_\alpha(b) ^{2^{-n}}.
 $$
Taking the $\liminf$ of the rhs, we finally obtain
$$ \vp(ab, ab) \leq \gamma_a \, \vp(b,b), \quad \forall b \in \Ao,$$
where
$\gamma_a:=\liminf_{n \to \infty} \left(  p_\alpha((a\ha a)^{2^n})\right)^{2^{-n}}$.
\end{proof}

\beex  As shown in Section 3, $\LDH [{\sf t}_s]$ is a $\mathfrak{P}$-regular partial *-algebra. The seminorms defining ${\sf t}_s$ satisfy \eqref{kapla}, since the elements of $\mathfrak{P}$ are bounded
operators in Hilbert space. Indeed, if $A \in {\mathfrak P}$,
$$
 \|(A\ha A)^{2^n}\xi\| \leq \|A\|^{2^{n+1}}\, \norm{\xi}, \,\forall \,\xi \in \D,
$$
 and so \eqref{kapla} holds in this case.
\enex

The following \emph{mixed} associativity in $\A$ , similar to \eqref{assoc},  can be easily
proved by using Definition \ref{strongprod}.
\begin{prop}\label{assoc3abs}
Let $x,y,z \in \A$. Assume that $x\wmult y$, $(x\wmult y)\wmult z$ and $y \bullet z$ are all well-defined. Then $x \in L(y\bullet z)$ and
$$
x\wmult(y\bullet z)= (x\wmult y) \wmult z.
$$
\end{prop}

As we have seen in Section
\ref{sect_preliminaries}, the $\tau^{\scriptscriptstyle\M} _{s^*}$-density of $\Ao$ and Proposition \ref{prop_415} imply the existence of  a strong
multiplication {\em induced} by $\Ao$. But this multiplication is, in general, only
a restriction of the multiplication $\bullet$ defined above. However, let us assume that  $\A$ is \emph{semi-associative with respect to
$\Ao$}, by which we mean that
\begin{equation} \label{eq_semiasswrtoazero}
(xa)b= x(ab); \quad a(xb)= (ax)b, \;\forall\, x \in \A, \, a,b \in \Ao.
\end{equation}
In other words, $(\A,\Ao)$ is a quasi *-algebra.
In that case,  Proposition \ref{assoc3abs} implies the topological regularity of $\A[\tau]$.

\begin{prop} Let $\A$ be semi-associative with respect to $\Ao$. Then $\A[\tau^{\scriptscriptstyle\M} _{s}]$ (and hence $\A[\tau]$) is topologically regular.
\end{prop}
\begin{proof} By Proposition \ref{prop_415}, the operator of left multiplication $\OL_x$ defined on $R(x)$  is $\tau$-closed, for every $x \in \A$. Let $y \in D({\ZOL_x})$, where $\ZOL_x$ denotes the closure of the restriction of $\OL_x$ to $\Ao$. Thus, there exists a net $\{y_\alpha\}\subset \Ao$ and $w \in \A$
such that $y_\alpha \stackrel{\tau^{\scriptscriptstyle\M} _{s}}{\to}y$ and $xy_\alpha \stackrel{\tau^{\scriptscriptstyle\M} _{s}}{\to}w$. Then, $x \in L(y)$ and, using ({\sf ips}$_4$),
\begin{align*} \vp((xy)a,z\ha b)&= \lim_\alpha \vp((xy_\alpha)a,z\ha b) = \lim_\alpha \vp(x(y_\alpha a),z\ha b)\\
&= \lim_\alpha \vp(y_\alpha a,(x\ha z\ha) b) = \vp(ya, (x\ha z\ha) b), \; \forall \,\vp \in \M, \, a,b \in \Ao.
\end{align*}
Hence ({\sf sm}$_1$) holds. The proof of ({\sf sm}$_2$) is similar.
\end{proof}

\berem If $\A$ is semi-associative with respect to $\Ao$, then $\Ao \subset R^{\rm s}\,\A$, the set of universal strong right multipliers of $\A$.
\enrem

An element $x$ has a {\em strong inverse} if there exists $x^{-1}\in \A$ such that $x\bullet x^{-1} = x^{-1} \bullet x = e$. The mixed associativity of Proposition \ref{assoc3abs} implies that, if a strong inverse  of $x$ exists, then it is unique.

\begin{theo}\label{thm_417}
Let $\A[\tau]$ be an $\Ao$-regular \pa\ satisfying condition {\sf (wp)} and let $\mathcal M$ be the set of all continuous ips-forms with core
$\Ao$. Let $\pi$ be a $(\tau,{\sf t}_{s})$-continuous *-representation
of $\A$, (that is $\pi:\A[\tau]\rightarrow\LDH[{\sf t}_{s}]$ continuously). Then, an element $x\in \A$ is $\mathcal M$-bounded if and only if $\pi(x)$ is a bounded operator.
\end{theo}
\begin{proof}
Let us define the following positive sesquilinear form:
$$
\varphi_\xi(x,y):=\ip{\pi(x)\xi}{\pi(y)\xi}.
$$
The conditions ({\sf ips}$_3$) and ({\sf ips}$_4$) are easily verified. By the continuity of $\pi$ we get
\begin{align*}
|\varphi_\xi(x,y)|&=|\ip{\pi(x)\xi}{\pi(y)\xi}| \leq \|\pi(x)\xi\|\|\pi(y)\xi\|
\\
 &\leq \gamma\,p_\alpha(x)p_\alpha(y),
\end{align*}
for some $\gamma >0$. Then $\vp_\xi$ is an ips- form and $\varphi_\xi\in\mathcal M$.

If $x$ is $\mathcal M$-bounded, by definition, we have:
$$
\varphi_\xi(xa,xa)\leq q_{\scriptscriptstyle\M}(x)^2\varphi_\xi(a,a),\quad\forall\, \xi\in\mathcal D,\, \forall\, a\in\Ao.
$$
For $a=e$, one has: $\varphi_\xi(x,x)=\|\pi(x)\xi\|^2\leq q_{\scriptscriptstyle\M}(x)\varphi_\xi(e,e)=q_{\scriptscriptstyle\M}(x)\|\xi\|^2$.

Conversely, let us suppose that $\pi(x)$ is bounded for all
$(\tau,{\sf t}_{s})$-continuous
*-representation $\pi$ of $\A$. In particular, the GNS
representation $\pi_\varphi$ defined by $\varphi\in\M$ is
$(\tau,{\sf t}_{s})$-continuous, so it is bounded on $\mathcal D_\varphi:=\{\lambda_\varphi(a),\, a\in\Ao\}$. Then, there exists
$\gamma >0$ such that $\|\pi_\varphi(x)\xi\|^2\leq \gamma ^2\|\xi\|^2,\;\forall\, \xi\in\mathcal D_\varphi$. Since
$\xi=\lambda_\varphi(a)$ with $a\in\Ao$, then
$\|\pi_\varphi(x)\lambda_\varphi(a)\|^2\leq \gamma ^2\|\lambda_\varphi(a)
\|^2, \;\forall \, a\in\Ao$, i.e., $\varphi(xa,xa)\leq \gamma ^2\varphi(a,a)$ and $x$ is $\mathcal M$-bounded.
\end{proof}

We expect that $\M$-bounded elements can also be characterized in terms of their spectral behavior. For this, some additional assumptions on the family $\M$ of ips-forms are needed.

\bedefi
\label{def:wbehaved}
Let $\M$ be a family of continuous ips-forms on $\A\times \A$. For every $\vp \in \M$, let $\pi_\vp$ denote the corresponding GNS representation. We say that $\M$ is {\em well-behaved} if
\begin{itemize}
  \item[({\sf wb}$_1$)] $\M$ is sufficient;
  \item[({\sf wb}$_2$)]  For every $\vp\in\M$ and every $a\in\A$, $\vp_a\in\M$ also, where $\vp_a(x,y): = \vp(xa,ya)$;
  \item[({\sf wb}$_3$)] If $x, y \in \A$ and $\pi_\vp(x)\mult \pi_\vp(y)$ is well-defined for every $\vp\in \M$, then there exists $z \in \A$ such that $\pi_\vp(x)\mult \pi_\vp(y) = \pi_\vp (z)$, for every $\vp \in \M$;
   \item[({\sf wb}$_4$)] $\A$ is $\tau^{\scriptscriptstyle\M} _{s^*}$-complete.
\end{itemize}
\findefi
To give an example, if $\MM = \LDH[{\sf t}_{s^\ast}]$ or if $\MM$ is any partial GC*-algebra of operators, the family
$$
\M := \{\psi_\xi :  \xi \in \D, \, \psi_\xi (X,Y) = \ip{X\xi}{Y\xi}, \, X,Y \in \MM\}
$$
is well-behaved.

\begin{prop} If $\M$ is well-behaved, then $\D(q_{\scriptscriptstyle\M})$ is a C*-algebra with the strong multiplication $\bullet$ and the norm $q_{\scriptscriptstyle\M}$.
\end{prop}
\begin{proof}
By Proposition \ref{prop_413} it follows that if $x,y\in \D(q_{\scriptscriptstyle\M})$, then $\pi_\vp(x)\mult \pi_\vp(y)$. Thus, by ({\sf wb}$_3$), there exists $z \in \A$ such that $\pi_\vp(x)\mult \pi_\vp(y) = \pi_\vp (z)$, for every $\vp \in \M$.
Then, for every $\vp \in \M$ and $a,b \in \Ao$,
\begin{align*}
\vp(ya,x\ha b) &= \ip{\pi_\vp(y)\lambda_\vp(a)}{\pi_\vp(x\ha)\lambda_\vp(b)}\\
&= \ip{\pi_\vp(x)\mult \pi_\vp(y)\lambda_\vp(a)}{\lambda_\vp(b)}\\
&= \ip{\pi_\vp(z)\lambda_\vp(a)}{\lambda_\vp(b)}\\
&= \vp(za,b).
\end{align*}
Hence $xy$ is well-defined and, by Proposition \ref{prop_412}, $x \bullet y$ is also well-defined.
Since $q_{\scriptscriptstyle\M}$ is a C*-norm on $\D(q_{\scriptscriptstyle\M})$, we need only to prove the completeness of $\D(q_{\scriptscriptstyle\M})$ to get the result.

Let $\{x_n\}$ be a Cauchy sequence with respect to the norm $q_{\scriptscriptstyle\M}$. Then
$\{x_n\ha\}$ is Cauchy too. Hence,  by  ({\sf wb}$_2$), for every $\vp \in \M$ and $a\in \Ao$ we have
$$
 \vp( (x_n- x_m)a, (x_n- x_m)a) \to 0, \; \mbox{as } n,m \to \infty
$$
and
$$
 \vp( (x_n\ha- x_m\ha)a, (x_n\ha- x_m\ha)a) \to 0, \; \mbox{as } n,m \to \infty.
$$
Therefore, $\{x_n\}$ is Cauchy also  with respect to $\tau^{\scriptscriptstyle\M} _{s^*}$. Then, by ({\sf wb}$_4$), there exists $x\in \A$ such that $x_n \stackrel{\tau^{\scriptscriptstyle\M} _{s^*}}{\to} x$.
Since
$$
 \vp(xa,xa)=\lim_{n \to \infty} \vp(x_na, x_na)\leq \limsup_{n \to \infty} q_{\scriptscriptstyle\M}(x_n)^2 \vp(a,a)
 $$
and $\limsup_{n \to \infty} q_{\scriptscriptstyle\M}(x_n)^2<\infty$ (by the boundedness of the sequence $\{q_{\scriptscriptstyle\M}(x_n)\}$), we conclude that $x$ is $\M$-bounded.
Finally, by the Cauchy condition, for every $\epsilon >0$, there exists $n_\epsilon \in {\mb N}$ such that, for every $n,m > n_\epsilon$, $q_{\scriptscriptstyle\M}(x_n-x_m) <\epsilon$. This implies that
$$
\vp( (x_n- x_m)a, (x_n- x_m)a) <\epsilon \vp(a,a), \quad \forall \vp \in \M, \, a \in \Ao.
$$
Then if we fix $n> n_\epsilon$ and let $m\to \infty$, we obtain
$$
 \vp( (x_n- x)a, (x_n- x)a) \leq \epsilon \vp(a,a), \quad \forall \vp \in \M, \, a \in \Ao.
$$
This, in turn, implies that $q_{\scriptscriptstyle\M}(x_n-x)\leq \epsilon$. This completes the proof.
\end{proof}

Let us now introduce the usual spectral elements adapted to the present situation.
\bedefi Let $x\in \A$.  The \emph{resolvent} $\rho^{\scriptscriptstyle\M} (x)$ of $x$ is defined by
$$
 \rho^{\scriptscriptstyle\M}(x) :=\left\{ \lambda\in {\mb C}: (x-\lambda e)^{-1} \mbox{exists in } \D(q_{\scriptscriptstyle\M}) \right\}.
$$
 The  corresponding \emph{spectrum} of $x$ is defined as $\sigma^{\scriptscriptstyle\M}(x):={\mb C}\setminus \rho^{\scriptscriptstyle\M}(x)$.
 \findefi

In similar way as in \cite{ctbound1} it can be proved that, if $\M$ is well-behaved, (a) $\rho^{\scriptscriptstyle\M}(x)$ is an open subset of the complex plane; (b) the map $\lambda \in \rho^{\scriptscriptstyle\M}(x)\mapsto (x-\lambda e)^{-1}\in \D(q_{\scriptscriptstyle\M})$ is analytic in each connected component of $\rho^{\scriptscriptstyle\M}(x)$; (c) $\sigma^{\scriptscriptstyle\M}(x)$ is nonempty.

As usual, we define the {\em spectral radius} of $x \in \A$ by
$$
r^{\scriptscriptstyle\M}(x):= \sup\{ |\lambda| : \, \lambda \in \sigma^{\scriptscriptstyle\M}(x)\}.
$$

\betheo Assume that $\M$ is well-behaved and let $x \in \A$. Then $r^{\scriptscriptstyle\M}(x)<\infty$ if and only if $x \in \D(q_{\scriptscriptstyle\M})$.
\entheo

\begin{proof}If $x \in \D(q_{\scriptscriptstyle\M})$, then $\sigma^{\scriptscriptstyle\M}(x)$ coincides with the spectrum of $x$ as an element of the C*-algebra $\D(q_{\scriptscriptstyle\M})$ and so $\sigma^{\scriptscriptstyle\M}(x)$ is compact.
Conversely, assume that $r^{\scriptscriptstyle\M}(x)<\infty$. Then the function $\lambda \mapsto
(x-\lambda e)^{-1}$ is $q_{\scriptscriptstyle\M}$-analytic  in the region
$|\lambda|> r^{\scriptscriptstyle\M}(x)$. Therefore it has there a $q_{\scriptscriptstyle\M}$-convergent Laurent expansion
$$
(x-\lambda e)^{-1} = \sum_{k=1}^\infty \frac{a_k}{\lambda^k},\quad |\lambda|> r^{\scriptscriptstyle\M}(x),
$$
 with $a_k \in \D(q_{\scriptscriptstyle\M})$  for each $k \in{\mb N}$. As usual
$$
 a_k = \frac{1}{2\pi i} \int_\gamma \frac{(x-\lambda e)^{-1}}{\lambda^{-k +1}} d\lambda, \qquad k \in {\mb N},
  $$
 where
$\gamma:=\{\lambda \in {\mb C}: |\lambda=R:\, R> r^{\scriptscriptstyle\M}(x)\}$
and the integral on the r.h.s. is meant to converge with respect to $q_{\scriptscriptstyle\M}$.

For every $\vp \in \M$ and $b,b' \in \Ao$, we have
\begin{eqnarray*}
\vp(a_k b, x\ha b') &=& \frac{1}{2\pi i} \int_\gamma \frac{\vp((x-\lambda e)^{-1}b, x\ha b')}{\lambda^{-k +1}} d\lambda
\\ &= &\frac{1}{2\pi i} \int_\gamma \frac{\vp((x-\lambda e)^{-1}b, (x\ha -\overline{\lambda}e) b')}{\lambda^{-k +1}} d\lambda
\\ &  & \;\; +\frac{1}{2\pi i}\int_\gamma \frac{\vp((x-\lambda e)^{-1}b, \overline{\lambda} b')}{\lambda^{-k +1}} d\lambda
\\&=& \frac{1}{2\pi i}\int_\gamma \frac{\vp(b,  b')}{\lambda^{-k +1}} d\lambda + \frac{1}{2\pi i} \int_\gamma \frac{\vp((x-\lambda
e)^{-1}b,  b')}{\lambda^{-k }} d\lambda
 \\& =&
\vp( a_{k+1}b, b').
 \end{eqnarray*}
 This implies that $xa_k$ is well defined, for every $k\in {\mb N}$ and $xa_k= a_{k+1}$.

 In particular,
\begin{eqnarray*}\vp(a_1 b, x\ha b') &=& \frac{1}{2\pi i} \int_\gamma \vp((x-\lambda
e)^{-1}b, x\ha b') d\lambda \\
&=& \frac{1}{2\pi i} \vp \left( \left( \int_\gamma (x-\lambda
e)^{-1} d\lambda\right) b, x\ha b'\right)\\
&=& \frac{1}{2\pi i}\vp(-b, x\ha b').
\end{eqnarray*}
Hence $xa_1=-x$. Thus finally $x=-a_2 \in \D(q_{\scriptscriptstyle\M})$.
\end{proof}

In our previous paper \cite{antratsc}, we have introduced a notion of strong inverse based on the multiplication obtained by closure,
and this has allowed us to derive a number of spectral properties. Now the notion of  strong multiplication $\bullet$ defined here (Definition \ref{strongprod2}) allows to obtain similar results. In particular, Proposition 4.13 of \cite{antratsc} may be generalized as follows.
\begin{prop}
Assume that $\A$ is is topologically regular over $\Ao$ and let $x \in \A$. Then, every $\lambda \in {\mb C}$
such that $|\lambda|>q_{\scriptscriptstyle\M}(x)$ belongs to $\rho^{\scriptscriptstyle\M}(x)$.
\end{prop}
\begin{proof}
Let $x^{-1}$ be the strong inverse by closure of $x\in\A$ so that $x^{-1}\in L(x) \cap R(x)$.
Of course, we may assume that $x \in \D(q_{\scriptscriptstyle\M})$.
Then the following analogue of  ({\sf sm}$_1$) holds true:
\begin{align*}
\vp((xx^{-1})a,z^*b)&=\vp(a,z^*b) = \vp(x^{-1}a,(x^*z^*)b), \\
&\hspace*{2cm} \forall \,z \in L(x),\forall\vp\in{\M}, \forall\, a,b\in\Ao.
\end{align*}
Let indeed $x^{-1}\in D(\ZOL_x)$. Then there exists a net  $\{w_{\alpha}\}\subset \Ao$ such that
$w_{\alpha}\stackrel{{\tau}_{s}^{\scriptscriptstyle\M}}{\to}x^{-1}, \;
 xw_{\alpha} \stackrel{{\tau}_{s}^{\scriptscriptstyle\M}}{\to} e$.
Then, using the the continuity of  $\vp\in\M$, that of the multiplication by $\Ao$, and ({\sf ips}$_4$), we have:
$$
\vp(x^{-1}a,(x^*z^*)b) = \lim_ \alpha \vp(w_{\alpha} a,(x^*z^*)b)=  \lim_ \alpha \vp((xw_{\alpha}) a,z^*b) = \vp(a,z^*b).
$$
In the same way, one proves the following following analogue of  ({\sf sm}$_2$)
$$
\vp(({x^{-1}}^*x^*)a,vb)=\vp(x^*a,(x^{-1}v)b), \,\forall\, v \in R(y),\forall\, \vp\in{\M},\forall\, a,b\in\Ao.
$$
Since $x\bullet x^{-1} = x^{-1}\bullet x = e$, one shows in the same way, for $x\in D(\ZOL_{x^{-1}})$, that
\begin{align*}
\vp((x^{-1}x)a,z^*b)&=\vp(a,z^*b) = \vp(x^{-1}a,({x^{-1}}^*z^*)b), \\
&\hspace*{2cm} \forall \,z \in L(x),\forall\vp\in{\M}, \forall\, a,b\in\Ao,
\end{align*}
and
$$
\vp((x^*{x^{-1}}^*)a,vb)=\vp({x^{-1}}^*a,xvb), \,\forall\, v \in R(y),\forall\, \vp\in{\M},\forall\, a,b\in\Ao.
$$
Thus we have proved that if $x^{-1}$ is the strong inverse by closure of $x\in\A$, as defined in \cite{antratsc},
then  $x^{-1}$ is also the strong inverse with respect
to the strong multiplication $\bullet$ (the converse is not true in general).

Combining this fact with Proposition 4.13 of \cite{antratsc}, we can conclude that $(x-\lambda e)^{-1}$ exists as a strong inverse,
which proves the statement.
\end{proof}

\berem The previous Proposition implies that, for every $x\in \A$, $r^{\scriptscriptstyle\M}(x) \leq q_{\scriptscriptstyle\M}(x)$, for every choice of the sufficient family $\M$. Clearly, if $x\not\in \D(q_{\scriptscriptstyle\M})$, then both $r^{\scriptscriptstyle\M}(x)$ and $q_{\scriptscriptstyle\M}(x)$ are infinite.
\enrem

\section{Existence of faithful representations}
\label{sec:represent}

The lesson of Theorem \ref{thm_417} is essentially that the notion of $\M$-bounded element given above is reasonable: as for the case of locally convex *-algebras, a good notion of boundedness of an element is equivalent to the boundedness of the operators representing it. This definition will be even more   significant if the locally convex partial *-algebra under consideration possesses \emph{sufficiently many} *-representations. This fact is expressed, in the case of locally convex *-algebras, through the notion of \emph{*-semisimplicity} which we will extend to locally convex partial *-algebras in a natural way.

A *-representation of a partial *-algebra $\A$ is a *-homomorphism  $\pi: \A\rightarrow\LDH$. If $\A[\tau]$ is $\Ao$-regular, then, by definition, it has a $\tau^*$-dense
distinguished *-algebra $\Ao$. Clearly, $\pi(\Ao)$ is a *-algebra of operators, but in general $\pi(\Ao) \not\subset \LD$.
However, we can always guarantee this property by changing the domain. Indeed:

\begin{prop} \label{indextension} Let $\A$ be an $\Ao$-regular \pa\
and let $\pi$  be a *-representation of $\A$ with domain $\D$ in $\H$. Put
$$\dd_1:=\left\{\xi_0+\sum_{i=1}^n\pi(b_i)\xi_i,\quad b_i\in\Ao,\xi_0, \ldots \xi_n \in\dd; n \in {\mb N}\right\}$$ and define
$$
\pi_1(a)\left(\xi_0+\sum_{i=1}^n\pi(b_i)\xi_i\right):=\pi(a)\xi_0 + \sum_{i=1}^n\pi(a)\mult\pi(b_i)\xi_i.
$$
Then $\pi_1$ is a *-representation of $\A$ with domain $\D_1\supset \D$ and $\pi(\Ao)  \subset \L\ad(\D_1)$.
\end{prop}
The proof of this proposition is given in Appendix A. Thus we can conclude that it is not restrictive to suppose that $\pi(\Ao)\subset\LD$.
\medskip

Now we can state the result announced at the end of Section \ref{sec:pmult-ips}.
\betheo Let $\A$ be an $\Ao$-regular \pa, with a sufficient family $\M$ of ips-forms, in particular, a partial GC*-algebra. Then:

(i) $\A$ has a faithful,  $(\tau,{\sf t}_{s})$-continuous representation into a partial GC*-algebra of operators.

(ii) Assume, in addition, that the family $\M$ is well-behaved.
Then $\A$ has a faithful,  $(\tau,{\sf t}_{s})$-continuous representation \emph{onto} a partial GC*-algebra of operators.
\entheo

\begin{proof}
(i) For every $\vp\in\M$ , let $(\pi_\vp, \lambda_\vp, \H_\vp)$ be the corresponding GNS  construction. Define, as usual,
$\H := \oplus_{\vp\in\M} \H_\vp$ and
$$
\D(\pi):= \{\xi = \big( \lambda_\vp(a) \big),\, a\in\Ao : \sum_{\vp\in\M} \norm{\pi_\vp(x)\lambda_\vp(a)}^2 < \infty, \,\forall\, x\in\A\}.
$$
Then, putting
$$
\pi(x)\xi:=  \big(\pi_\vp(x)\lambda_\vp(a)\big),\, a\in \Ao,
$$
one defines a faithful representation of $\A$.

Taking into account the continuity of  $\vp\in\M$ and of the multiplication by $\Ao$, we have:
$$
\norm{\pi_\vp(x)\lambda_\vp(a)}^2 = \vp(xa,xa) \leq p(xa)^2 \leq  p'(xa),
$$
for some $\tau$-continuous seminorms $p,p'$. This implies that $\pi$ is  $(\tau,{\sf t}_{s})$-continuous. So, by Theorem \ref{thm_417},
if $x\in \D(q_{\scriptscriptstyle\M}), \, \pi(x)$ is bounded and one finds by a direct check that
$$
\norm{\overline{\pi(x)}}\leq q_{\scriptscriptstyle\M}(x).
$$

(ii) Let now the family $\M$ be well-behaved. Then $ \D(q_{\scriptscriptstyle\M})$ is a C$\ha$-algebra and, hence,
$$
\norm{\overline{\pi(x)}} = q_{\scriptscriptstyle\M}(x), \; \forall\, x\in  \D(q_{\scriptscriptstyle\M}),
$$
and  $\pi( \D(q_{\scriptscriptstyle\M}))$ is a C$\ha$-algebra.

Moreover, if $\Ao \subset\D(q_{\scriptscriptstyle\M}) $, then $\D(q_{\scriptscriptstyle\M})$ is $\tau^*$-dense in $\A$.
Hence, if $x\in\A$, there exists a net $\{x_\alpha\}\subset \Ao$ such that $x_\alpha \stackrel{\tau^*}{\to} x$. This implies that
$x_\alpha \stackrel{\tau}{\to} x$ and $x_\alpha\ha \stackrel{\tau}{\to} x\ha$.

Then, since $\pi$ is  $(\tau,{\sf t}_{s})$-continuous, we have that $\pi(x_\alpha)\xi \to \pi(x)\xi$ and $\pi(x_\alpha\ha)\xi \to \pi(x\ha)\xi,
\, \forall\, \xi\in \D(\pi)$. This implies that $\pi(x_\alpha)\xi  \stackrel{{\sf t}_{s^\ast}}{\to} \pi(x)\xi$.
Hence, $\pi(\D(q_{\scriptscriptstyle\M}))$ is ${\sf t}_{s^\ast}$-dense in $\pi(\A)$.

The construction of  $\pi$ implies that  $\pi(\A)$ is a \pa. Assume indeed that $\pi(x)\mult\pi(y)$ is well-defined. Then
$\pi_\vp(x)\mult\pi_\vp(y)$ is well-defined, for every  $\vp\in\M$. Hence there exists a $z\in\A$ such that
$\pi_\vp(x)\mult\pi_\vp(y) = \pi_\vp(z)$. This in turn implies that $\pi(x)\mult\pi(y) = \pi(z)$.

In general, however, $\pi(\A)$ need not be complete with respect to ${\sf t}_{s^\ast}$. Assume that $\{\pi(x_\alpha)\}$ is a net in  $\pi(\A)$:
$$
\pi(x_\alpha) \stackrel{{\sf t}_{s^\ast}}{\to}Z\in {\L}\ad(\D(\pi),\H).
$$
Then, by the definition of $\pi$,
$$
\pi_\vp(x_\alpha) \stackrel{{\sf t}_{s^\ast}}{\to}Z_\vp, \; \mbox{ where } Z_\vp \xi_\vp = \big(Z(\xi_\vp)\big)_\vp.
$$
This implies that, for every $a\in\Ao$,
\begin{align*}
&\pi_\vp(x_\alpha)  \lambda_\vp(a) \to Z_\vp \lambda_\vp(a),
\\
&\pi_\vp(x_\alpha\ha)  \lambda_\vp(a) \to Z_\vp\ha \lambda_\vp(a).
\end{align*}
Hence
$$
\vp((x_\alpha -x_\beta)a, (x_\alpha -x_\beta)a) = \norm{\pi_\vp(x_\alpha -x_\beta)\lambda_\vp(a)}^2 \to 0,
$$
for $\alpha,\beta$ ``large" enough and every $a\in\Ao$. Similarly,
$$
\vp((x_\alpha\ha -x_\beta\ha)a, (x_\alpha\ha -x_\beta\ha)a) = \norm{\pi_\vp(x_\alpha\ha -x_\beta\ha)\lambda_\vp(a)}^2 \to 0,
$$
Since $\M$ is well-behaved, $\{x_\alpha\}$ is a $\tau_{s^*}^{\scriptscriptstyle\M}$-Cauchy net. Thus there exists $x\in\A$ such that
$$
\vp(x_\alpha-x, x_\alpha-x) \to 0, \; \forall\, \vp\in\M.
$$
By ({\sf wb}$_2$), it follows that
$$
\vp((x_\alpha-x)a, (x_\alpha-x)a) \to 0, \; \forall\, \vp\in\M, \, a\in \Ao.
$$
In conclusion, $\pi_\vp(x_\alpha)  \stackrel{{\sf t}_{s^\ast}}{\to}\pi_\vp(x), \,  \forall\, \vp\in\M$ and, hence,
 $\pi(x_\alpha)  \stackrel{{\sf t}_{s}}{\to}\pi(x)$. Thus $\pi(\A)$ is ${\sf t}_{s^\ast}$-closed and, hence,
 $\pi(\A)$ is ${\sf t}_{s^\ast}$-complete, if one remembers that $\LDH[{\sf t}_{s^\ast}]$ is complete.
 This concludes the proof.
\end{proof}

\medskip

For topological *-algebras the set of elements which belong to the intersection of the kernel of all continuous *-representations constitute
the so-called *-radical of $\A$ (see, e.g. \cite{bonsall, palmer}).

In a previous paper \cite{anttratschi_2}, we have introduced the notion of algebraic \mbox{*-radical}  and
the attending definition of an algebraically *-semisimple partial *-algebra.
In the present context, the presence of a sufficient family of continuous ips-forms  allows one to introduce similar concepts at the topological level as well. Thus the notion of *-radical has a natural extension to our case.

Let in fact $\A[\tau]$ be an $\Ao$-regular partial *-algebra. We define the \emph{*-radical} of $\A$ by:
$$
\R\ha(\A):=\{x\in\A:\, \pi(x)=0,\, \mbox{for all}\, (\tau,{\sf t}_{s^*})\mbox{-continuous *-representations}\,\, \pi\}
$$
We put $\R^*(\A):=\A$, if $\A[\tau]$ has no $(\tau,{\sf t}_{s^*})\mbox{-continuous *-representations}$.

\begin{prop}\label{prop_finalnew}
Let $\A[\tau]$ be an $\Ao$-regular partial *-algebra and $\P_{\Ao}(\A)$   the set of \emph{all} $\tau$-continuous ips-forms with core $\Ao$. For an element $x\in \A$ the following statements are equivalent.
\begin{itemize}
  \item[(i)] $x \in \R^*(\A)$.
  \item[(ii)] $\vp(x,x)=0$ for every $\vp \in \P_{\Ao}(\A)$.
  \item[(iii)] $x\ha x$ is well-defined and $x\ha x=0$.
\end{itemize}
\end{prop}
\begin{proof}
(i)$\Rightarrow$(ii): Assume that, for all $x\in\A, x\neq 0$, there exists a continuous ips-form with core $\Ao$, such that $\varphi(x,x)>0$.
Let $(\pi_\varphi,\hh_\varphi,\lambda_\varphi)$
 be the corresponding GNS construction. The GNS *-representation is $(\tau^*,{\sf t}_{s^*})$-continuous. Indeed, if $a\in\Ao$, we have:
$$
\|\pi_\varphi(x)\lambda_\varphi(a)\|^2=\varphi(xa,xa)\leq\gamma^2p_\alpha^*(xa)\leq\gamma 'p_\beta^*(x)\,.
$$
On the other hand,
$$\|\pi_\varphi(x^*)\lambda_\varphi(a)\|^2=\varphi(x^*a,x^*a)\leq\gamma^2p_\alpha^*(x^*a)\leq\gamma ''p_\beta^*(x^*)=\gamma'''p_\beta^*(x).$$
Finally,
$\|\pi_\varphi(x)\lambda_\varphi(e)\|^2=\varphi(x,x)>0,$
and this implies $\pi_\varphi(x)\neq 0$.

(ii)$\Rightarrow$(iii) Assume that $\varphi(x,x)=0$ for all $\varphi\in\P_{\Ao}(\A)$. For $a\in\Ao$,  $\varphi_a(x,x)=0$, since as it is easy to see $\varphi_a\in\P_{\Ao}(\A)$. By the Cauchy-Schwarz inequality, it follows that $\varphi(xa,xb)=0$, for all $a,b\in\Ao$. By
{\sf (wp)}, this means that $x^*\mult x =x^*x  $ is well defined and  $x^* x=0$.

(iii)$\Rightarrow$(i) Assume now that $x\ha x$ is well-defined and $x^*x=0$. Then, if $\pi$ is a $(\tau,{\sf t}_{s^*})$-continuous *-representation of $\A$, $\pi(x\ha)\mult\pi(x)=\pi(x)\ad\mult\pi(x)$ is well-defined and equals $0$ . Hence, for every $\xi \in \D(\pi)$,
\begin{align*} \|\pi(x)\xi\|^2 &= \ip{\pi(x)\xi}{\pi(x)\xi}= \ip{\pi(x)\xi}{\pi(x)\xi} \\
&=\ip{\pi(x)\ad\mult\pi(x) \xi}{\xi}= \ip{\pi(x\ha)\mult\pi(x) \xi}{\xi}\\ &=\ip{\pi(x\ha x)\xi}{\xi} =0.
\end{align*}
Hence $\pi(x) = 0$.

\end{proof}

Clearly if $\A$ possesses a sufficient family $\M$ of $\tau$-continuous ips-forms, then $\P_{\Ao}(\A)$ itself is sufficient. By Proposition \ref{prop_finalnew}, it follows that $\R^*(\A)=\{0\}$. Conversely, if $\R^*(\A)=\{0\}$, then $\P_{\Ao}(\A)$ is sufficient. The choice of considering a sufficient family $\M$ instead of the whole $\P_{\Ao}(\A)$ is motivated by the fact that characterizing the space $\P_{\Ao}(\A)$ in concrete examples is much more difficult than choosing a sufficient subfamily.

 As for the case of topological algebras, it is  natural, at the light of the previous discussion, to call {\em *-semisimple} an $\Ao$-regular partial *-algebra $\A[\tau]$ such that $\R\ha(\A)=\{0\}$. We hope to carry out a more detailed analysis of this situation in a further paper.

\appendix
\section{Proof of Proposition \ref{indextension}}
We give here a proof of Proposition \ref{indextension}. The argument used is very similar to that given in \cite[Proposition 1]{schm_wb} in a different context.
To keep the notation lighter, we will assume that $\pi(e)= I_\D$. The general case can be proved by a slight modification of the argument below. Note that all the considered sums are finite.

We have to check that $\pi_1(a)$ is well-defined  for every $a \in \A$ and that $\pi_1$ is a *-representation of $\A$.
\begin{align*}
\ip{\sum_i(\pi(a)\mult\pi(b_i))\xi_i}{\sum_j\pi(c_j)\eta_j}
&=\sum_{i,j}\ip {\pi(ab_i)\xi_i}{\pi(c_j)\eta_j}\\
&\hspace*{-3.5cm}=\sum_{i,j}\ip {\xi_i}{\pi(ab_i)\ad\pi(c_j)\eta_j}
\\&\hspace*{-3.5cm}
=\sum_{i,j}\ip {\xi_i}{\pi(b_i^*a^*)\mult\pi(c_j)\eta_j}
=\sum_{i,j}\ip {\xi_i}{\pi((b_i^*a^*)c_j)\eta_j}
\\&\hspace*{-3.5cm}
=\sum_{i,j}\ip {\xi_i}{(\pi(b_i^*)\mult\pi(a^*c_j))\eta_j}
= \sum_{i}\ip {\pi(b_i)\xi_i}{\sum_{j}(\pi(a^*c_j)\eta_j}.
\end{align*}
Hence, if $\sum_i \pi(b_i)\xi_i =0$, then $\xi:= \sum_i(\pi(a)\mult\pi(b_i))\xi_i$ is orthogonal to every element of $\D_1$, which is dense in $\H$.
Thus $\xi=0$.
This proves that $\pi_1(a)$ is, for every $a\in\A$, a well-defined linear map of $\dd_1$ into $\hh$. Clearly, $\pi_1(\Ao)\subset \LD$.\\
Moreover, the above equalities also imply that
$$\ip{\pi_1(a)\left(\sum_i\pi(b_i)\xi_i\right)}{\sum_j\pi(c_j)\eta_j}=\ip {\sum_i\pi(b_i)\xi_i}{\pi_1(a^*)\sum_j\pi(c_j)\eta_j}.$$
Hence, $\pi_1(a)\ad=\pi_1(a^*)$.\\
Let now $a_1,a_2\in\A$ with $a_1a_2$ well-defined. We have to prove that $\pi_1(a_1)\mult\pi_1(a_2)$ is well-defined and $\pi_1(a_1)\mult\pi_1(a_2)=\pi_1(a_1a_2)$:
\begin{align*}
&\ip  {\pi_1(a_1a_2)(\sum\pi(b_i)\xi_i)}{\sum\pi(c_j)\eta_j}=
\sum_{i,j}\ip {\pi(a_1a_2)\mult\pi(b_i)\xi_i}{\pi(c_j)\eta_j}.
\end{align*}
On the other hand:
\begin{align*}
\ip  {\pi_1(a_2)(\sum\pi(b_i)\xi_i)}{\pi_1(a_1)\ad\sum\pi(c_j)\eta_j}&=
\sum_{i,j}\ip {\pi(a_2)\mult\pi(b_i)\xi_i}{\pi(a_1^*)\mult\pi(c_j)\eta_j}
\\
&=\sum_{i,j}\ip {\pi(a_2b_i)\xi_i}{\pi(a_1^*c_j)\eta_j}.
\end{align*}
Now, checking that $c_j^*((a_1a_2)b_i)=(c_j^*(a_1a_2))b_i$, we have:
\begin{align*}
&\sum_{i,j}\ip {\pi(a_2b_i)\xi_i}{\pi(a_1^*c_j)\eta_j}=\sum_{i,j}\ip {\pi(c_j^*a_1)\mult\pi(a_2b_i)\xi_i}{\eta_j}\\
&=
\sum_{i,j}\ip {(\pi\ad(c_j)\mult\pi(a_1))\mult\pi(a_2b_i)\xi_i}{\eta_j}
=\sum_{i,j}\ip {\pi\ad(c_j)\mult(\pi(a_1)\mult\pi(a_2b_i))\xi_i}{\eta_j}\\
&=
\sum_{i,j}\ip {(\pi(a_1)\mult\pi(a_2b_i))\xi_i}{\pi(c_j)\eta_j}=\sum_{i,j}\ip {((\pi(a_1)\mult\pi(a_2))\mult\pi(b_i)\xi_i}{\pi(c_j)\eta_j}\\
&=
\ip {(\sum\pi(a_1a_2))\mult\pi(b_i)\xi_i}{\sum\pi(c_j)\eta_j}=\ip {(\pi_1(a_1a_2))\sum\pi(b_i)\xi_i}{\sum\pi(c_j)\eta_j}.
\end{align*}
This proves the statement.

\end{document}